\documentclass[11pt]{article}

\usepackage{amsmath,amssymb,amsthm,amscd}
\usepackage{pictexwd}
\usepackage{latexsym}
\usepackage[all]{xy}
\usepackage{graphicx}
\usepackage[dvipsnames,usenames]{color}

 \theoremstyle{plain}
\newtheorem{thm}{Theorem}[section]
\newtheorem{lemma}[thm]{Lemma}
\newtheorem{prop}[thm]{Proposition}
\newtheorem{cor}[thm]{Corollary}

\theoremstyle{definition}

\newtheorem{example}{Example}

\xyoption{dvips}
\numberwithin{equation}{section}
\newcommand{\FF}{\mathbb{F}}  
\newcommand{\ZZ}{\mathbb{Z}}  
\newcommand{\CC}{\mathbb{C}}  
\newcommand{\Des}{\mathit{Des}}  
\newcommand{\BS}{\mathit{BS}} 
\newcommand{\maj}{\mathit{maj}}  
\newcommand{\imaj}{\mathit{imaj}}
\newcommand{\inv}{\mathit{inv}}
\newcommand{\shape}{\mathit{shape}}
\newcommand{\rightmost}{\mathrm{RT}}  
\newcommand{\GL}{\mathrm{GL}}
\newcommand{\Orthog}{\mathrm{O}}
\newcommand{\End}{\mathrm{End}}    
\newcommand{\Resf}{\mathrm{Resf}}  
\newcommand{\Indf}{\mathrm{Indf}}  
\newcommand{\Res}{\mathrm{Res}}  
\newcommand{\Ind}{\mathrm{Ind}}  

\newcommand{\IR}{\mathcal{IR}}

\newcommand{\One}{{1\hspace{-.14cm} 1}}
\newcommand{\spanning}{\text{-span}}

\newcommand{\DI}{{\buildrel\rm{DI} \over \longrightarrow}}
\newcommand{\RSK}{\ {\buildrel\rm{RSK} \over \longleftarrow}\ }
\newcommand{\jdt}{\ {\buildrel\rm{jdt} \over \longrightarrow}\ }

\title{$q$-Partition Algebra Combinatorics}
\author{
Tom Halverson \\
Department of Mathematics \\
Macalester College \\
Saint Paul, MN 55105 \\
halverson@macalester.edu 
\and Nathaniel Thiem \\
Department of Mathematics \\
University of Colorado at Boulder \\
Boulder, CO 80309 \\
nathaniel.thiem@colorado.edu
}

\begin{document}
\maketitle

\begin{abstract}
We study  a $q$-analog $Q_r(n,q)$  of the partition algebra $P_r(n)$.  The algebra $Q_r(n,q)$  arises as the centralizer algebra  of the finite general linear group $\GL_n(\FF_q)$ acting on a vector space $\IR_q^r$ coming from $r$-iterations of Harish-Chandra restriction and induction.  For $n \ge 2r$, we show that $Q_r(n,q)$ has the same semisimple matrix  structure as $P_r(n)$.  We compute the dimension $d_{n,r}(q) = \dim(\IR_q^r)$ to be a $q$-polynomial that specializes as $d_{n,r}(1) = n^r$ and $d_{n,r}(0) = B(r)$, the $r$th Bell number. Our method is to write $d_{n,r}(q)$  as a sum over integer sequences which are $q$-weighted by inverse major index.  We then find a basis of $\IR_q^r$  indexed by $n$-restricted $q$-set partitions of $\{1,\ldots, r\}$ and show that there are $d_{n,r}(q)$ of these.
\end{abstract}

\bigskip

\section*{Introduction}

The general linear group $\GL_n(\CC)$ and the symmetric group $S_r$ both act on tensor space 
$V^{\otimes r}$, where $V$ is the natural $n$ dimensional representation of $\GL_n(\CC)$ and $S_r$ acts by tensor place permutations.
Classical Schur--Weyl duality says that these actions commute and that each action generates the full centralizer of the other, so that as a $(\GL_n(\CC), S_r)$-bimodule the tensor space has a multiplicity-free decomposition given by 
$
V^{\otimes r} \cong \bigoplus_\lambda L(\lambda) \otimes S_r^\lambda,
$
where the $L(\lambda)$ are irreducible $\GL_n(\CC)$-modules and the $S_r^\lambda$ are the irreducible $S_r$-modules.

If we restrict $\GL_n(\CC)$ to its subgroup of orthogonal matrices $\Orthog_n(\CC)$, then the centralizer algebra is Brauer's centralizer algebra $B_r(n)= \End_{\Orthog_n(\CC)}(V^{\otimes r})$.   If we restrict further to the symmetric groups $S_{n-1} \subseteq S_n \subseteq \Orthog_n(\CC) \subseteq \GL_r(\CC)$, then the centralizer algebras are the partition algebras $P_r(n) = \End_{S_n}(V^{\otimes r})$ and  $P_{r + \frac{1}{2}}(n) = \End_{S_{n-1}}(V^{\otimes r})$.   Furthermore, the containments reverse:
 $$
 \begin{array}{rcccccccc}
\hbox{subgroup } G \subseteq \GL_n(\CC): \quad & S_{n-1} & \subseteq & S_n & \subseteq & \Orthog_n(\CC) & \subseteq & \GL_ n(\CC) \\
 & \updownarrow   && \updownarrow  && \updownarrow \\
\hbox{centralizer algebra } \End_G(V^{\otimes r}): \quad & P_{r + \frac{1}{2}}(n) & \supseteq  & P_r(n) & \supseteq & B_r(n) & \supseteq & \CC S_r. \\
\end{array}
$$
The Brauer algebras were introduced in 1937 by Richard Brauer.  The partition algebras arose early in the 1990s in the work of Martin \cite{Mar1}, \cite{Mar2} and later, independently, in the work of Jones  \cite{Jo} (see also \cite{HR}).  For $r \in \frac{1}{2} \ZZ_{>0}$, the partition algebra $P_r(n)$ has a basis indexed by the set partitions of $\{1, 2, \ldots, 2 r\}$ and a multiplication given by ``diagram multiplication."

If the groups $\GL_n(\CC)$ and $\Orthog_n(\CC)$ are replaced by their quantum groups $U_q ({\mathfrak{gl}}_n)$ and $U_q ({\mathfrak{o}}_n)$, respectively, then their centralizer algebras become 
 $$
 \begin{array}{rcccc}
\hbox{quantum group}: \quad  & U_q ({\mathfrak{o}}_n) & \subseteq & U_q ({\mathfrak{gl}}_n) \\
    & \updownarrow  && \updownarrow \\
\hbox{centralizer algebra}: \quad & BMW_{r}(n,q) & \supseteq & H_r(q). \\
\end{array}
$$
The algebra $H_r(q)$ is the Iwahori--Hecke algebra of the symmetric group $S_r$, and its action on tensor space $V^{\otimes r}$ is due to Jimbo \cite{Ji}. The algebra $BMW_r(n,q)$ is a $q$-analog of the Brauer algebra, called the Birman--Murakami--Wenzl algebra \cite{BW}, \cite{Mur}.  This leads naturally to the problem of finding a $q$-analog of the partition algebra.  One might consider replacing the symmetric group $S_n$ with its Iwahori-Hecke algebra $H_n(q)$, but $H_n(q)$ does not have a Hopf coproduct allowing it to act on the tensor product representation $V^{\otimes r}$. Furthermore, the Hecke algebra $H_n(q)$ is not found as a subalgebra of $U_q ({\mathfrak{o}}_n)$ (in fact, even the containment of $U_q ({\mathfrak{o}}_n)$ is  $U_q ({\mathfrak{gl}}_n)$ is more subtle than  $\Orthog_n(\CC) \subseteq \GL_n(\CC)$).

In this paper, we take a different approach to defining a $q$-partition algebra in which we replace the underlying tensor space with a module constructed by iterations of restriction and induction of finite general linear group modules.  This approach was first proposed  in unpublished work of T.\ Halverson and A.\ Ram and is motivated by the analogous construction of the partition algebra through restriction and induction of symmetric group modules (see \cite{HR}).  A forthcoming paper by T.\ Halverson, A.\ Ram, and N.\ Thiem will further study the $q$-partition algebra, and the analysis of the underlying restriction-induction module found in this paper is essential to that work.

\bigskip\noindent
The paper is organized as follows:

\begin{enumerate}
\item In Section 1, we describe the construction of a $q$-partition algebra $Q_r(n,q)$ as the centralizer of the general linear $\GL_n(\FF_q)$ over a finite field $\FF_q$ having $q$ elements.  For $n \ge 2 r$,  $Q_r(n,q)$ is the centralizer of $\GL_n(\FF_q)$ acting a vector space $\IR_q^r$, consisting of $r$ iterations of Harish-Chandra restriction and induction.  At $q = 1$, we have $\IR_{1}^r \cong V^{\otimes r}$, and we think of the symmetric group $S_n$ as the $q \to 1$ limit of $\GL_n(\FF_q)$.   We show that $Q_r(n,q)$ and $P_r(n,q)$ each have dimension equal to the Bell number $B(2r)$, and that for $n \ge 2r$ they have the same matrix block structure as semisimple algebras. 

\item  In  Section 2, we combinatorially study the dimension of of $\IR_q^k$. We show that 
\begin{equation}\label{intro:Formula}
\dim(\IR_q^r) = d_{n,r}(q) = \sum_{\ell = 1}^n S(r,\ell) [n] [n-1] \cdots [n-\ell+1],
\end{equation}
where $S(r,\ell)$ is a Stirling number of the second kind and $[j] = (q^j-1)/(q-1)$ is a $q$-integer.  The $q$-polynomial  $d_{n,r}(q)$ that appears in this formula has the property that $d_{n,r}(1) = n^r$ and $d_{n,r}(0) = B(r)$, the $r$th Bell number or number of partitions of $\{1, \ldots, r\}$ into subsets.  Thus $d_{n,r}(q)$ is a $q$-analog of both $n^r$ and $B(r)$, and it interpolates between the two as $q$ ranges through $0 \le q \le 1$. 
Our method is to write the dimension as a $q$-weighted sum over sequences $a= (a_1, \ldots, a_r)$, where each sequence is weighted by an analog of the inverse major  index $\imaj(a)$.  This is done in Proposition \ref{propsum}.   We then use a Schensted bijection (see \eqref{InsertionDef})  and the decomposition of $\IR_q^r$ as a $( \GL_n(\FF_q), Q_r(n,q))$-bimodule to prove formula \eqref{intro:Formula}.

\item  In Section 3, we define ($n$-restricted) $q$-set partitions of $\{1, \ldots, r\}$, and we show that $d_{n,r}(q)$ enumerates these objects.  We study the module $\IR_q^r$ and  find a basis for it that is indexed by these $q$-weighted set partitions of $\{1, \ldots, r\}$.  
\end{enumerate}

T.\ Halverson was partially supported by the National Science Foundation under grant DMS-0100975.  This research was completed while the authors were in residence at the Mathematical Sciences Research Institute (MSRI) in Spring 2008 for the program in Combinatorial Representation Theory. We are grateful for the support and the stimulating research environment at MSRI.  We thank Arun Ram for numerous useful conversations, and we thank Vic Reiner and Dennis Stanton for a helpful conversation about the distribution of the statistics $\inv$ and $\imaj$ used in the proof of Proposition  \ref{propsum}.  We are grateful to the anonymous referees for helpful suggestions.

\begin{section}{A $q$-Partition Algebra}

View $S_{n-1} \subseteq S_n$ under the natural embedding. The $n$-dimensional permutation module $V$ for $S_n$ is isomorphic to the induced module,  $V \cong \Ind_{S_{n-1}}^{S_n}(\One_{n-1}) = \CC S_n \otimes_{S_{n-1}} \One_{n-1}$, where $\One_{n-1}$ is the trivial $S_{n-1}$-module.  In \cite{HR}, Halverson and Ram emphasize viewing tensor products of  $S_n$-modules via restriction and induction and the ``tensor identity,"
\begin{equation}
\begin{array}{ccc}
\Ind_{S_{n-1}}^{S_n} \Res_{S_{n-1}}^{S_n} (W) &\cong& W \otimes V  \\
g \otimes_{S_{n-1}} w & \mapsto & g m \otimes ( g \otimes_{S_{n-1}} \One_{n-1}),
\end{array}
\end{equation}
where $W$ is any $S_n$ module, $w \in W, g \in S_{n-1}$. For $r \in \ZZ_{\ge 0}$ define,
\begin{equation}\label{ResIndSpace}
\IR^r_1 =  \underbrace{\Ind_{S_{n-1}}^{S_n} \Res_{S_{n-1}}^{S_n} \cdots \Ind_{S_{n-1}}^{S_n} \Res_{S_{n-1}}^{S_n}}_{\hbox{$2r$ functors}}  (\One_{n}), 
\end{equation}
and
\begin{equation}\label{ResIndSpaceHal}
\IR^{r + \frac{1}{2}}_1 =  \underbrace{ \Res_{S_{n-1}}^{S_n} \Ind_{S_{n-1}}^{S_n} \Res_{S_{n-1}}^{S_n} \cdots \Ind_{S_{n-1}}^{S_n} \Res_{S_{n-1}}^{S_n}}_{\hbox{$2r+1$ functors}}  (\One_{n}), 
\end{equation}
Then, by induction, $V^{\otimes r}  \cong \IR_1^r$,  and therefore the partition algebra satisfies, for $n \ge 2r$,
\begin{equation}\label{PnResInd}
\begin{array}{rclcl}
P_r(n) &=& \End_{S_n}( V^{\otimes r} )&\cong& \End_{S_n}(  \IR_1^r ), \\
P_{r + \frac{1}{2}}(n) &=& \End_{S_{n-1}}( V^{\otimes r} )&\cong& \End_{S_{n-1}}(  \IR_1^{r + \frac{1}{2}} ).
\end{array}
\end{equation}

Key to the decomposition of $\IR_1^r$ into irreducible symmetric group modules are the restriction and induction rules
\begin{equation}\label{ResIndForSn}
\Res_{S_{n-1}}^{S_n} (S_n^\lambda)  = \bigoplus_{\mu = \lambda - \square} S_{n-1}^\mu \quad\hbox{ and }\quad
\Ind_{S_{n-1}}^{S_n} (S_{n-1}^\mu)  = \bigoplus_{\lambda = \mu + \square} S_{n}^\lambda,
\end{equation}
where $S_n^\lambda$ is the irreducible $S_n$ module labeled by the partition $\lambda \vdash n$, 
$S_{n-1}^\mu$ is the irreducible $S_{n-1}$ module labeled by the partition $\mu \vdash (n-1)$, and $\lambda - \square$ and $\mu +\square$ denote adding and removing a box from the partition, respectively.

One can view the symmetric group $S_n$ as the $q\to1$ limit of the general linear group $G_n = \GL_n(\FF_q)$ over the finite field $\FF_q$.   Indeed, if $B$ is the Borel subgroup of upper triangular matrices in $\GL_n(\FF_q)$, then 
$$
| \GL_n(\FF_q)/B | = [n][n-1] \cdots [2][1],
$$
where $[n] = 1 + q + \ldots q^n$ is a $q$-analog of $n$ so that $[n][n-1] \cdots [2][1]$ is a $q$-analog of $n!$.  Furthermore, the irreducible unipotent representations of $G_n$ are denoted $G_n^\lambda$ and labeled by partitions $\lambda \vdash n$  (see, for example, \cite[\S 4.3]{Mac}).

We view $G_{n-1} \subseteq G_n$ as a Levi subgroup with blocks of size 1 and $n-1$ (see Section 3.3). 
Under Harish-Chandra restriction $\Resf$ and induction $\Indf$  (see Section 3.3) these modules satisfy exactly the same rules as  \eqref{ResIndForSn}, namely (see \cite[\S 4.3]{Mac}),
\begin{equation}\label{ResIndForGn}
\Resf_{G_{n-1}}^{G_n} (G_n^\lambda)  = \bigoplus_{\mu = \lambda - \square} G_{n-1}^\mu \quad\hbox{ and }\quad
\Indf_{G_{n-1}}^{G_n} (G_{n-1}^\mu)  = \bigoplus_{\lambda = \mu + \square} G_{n}^\lambda.
\end{equation}
For $r \in \ZZ_{\ge 0}$ define,
\begin{equation}\label{qResIndSpace}
\IR^r_q =  \underbrace{\Indf_{G_{n-1}}^{G_n} \Resf_{G_{n-1}}^{G_n} \cdots \Ind_{G_{n-1}}^{G_n} \Res_{G_{n-1}}^{G_n}}_{\hbox{$2r$ functors}}  (\One_{G_n}), 
\end{equation}
and
\begin{equation}\label{qResIndSpaceHalf}
\IR^{r + \frac{1}{2}}_q =  \underbrace{ \Res_{G_{n-1}}^{G_n} \Ind_{G_{n-1}}^{G_n} \Res_{G_{n-1}}^{G_n} \cdots \Ind_{G_{n-1}}^{G_n} \Res_{G_{n-1}}^{G_n}}_{\hbox{$2r+1$ functors}}  (\One_{G_n}), 
\end{equation}
Then, for integers $n \ge 2r$,  define
\begin{equation}\label{QnResInd}
\begin{array}{rcl}
Q_r(n,q) &=& \End_{G_n}(  \IR_q^r ), \\
Q_{r + \frac{1}{2}}(n,q) &=& \End_{G_{n-1}}(  \IR_q^{r + \frac{1}{2}} ).
\end{array}
\end{equation}
Equation \eqref{QnResInd} completely defines $Q_r(n,q)$ as an algebra of endomorphisms, however considerable work needs to be done to find a natural set of generators for $Q_r(n,q)$ and the relations that they satisfy.    This will be the subject of a forthcoming paper by T.\ Halverson, A.\ Ram, and N.\ Thiem.  The analysis in this paper will be foundational to that work.

The \emph{Bratteli diagram}  $\mathfrak{B}_{n}$ is a graph that encodes the decomposition of $\IR_q^r$.  Let $\mathfrak{B}_{n}$ have vertices organized into levels indexed by $r \in \frac{1}{2} \ZZ_{\ge 0}$  such that the vertices on level $r$ are labeled by the set of integer partitions $\Lambda_n^r$ defined by
\begin{equation}\label{IndexSet}
\begin{array}{rcl}
\Lambda_n^r& =& \left\{ \  \lambda = (\lambda_1, \ldots, \lambda_t) \vdash  n \ \big\vert \  \lambda_2 + \cdots +  \lambda_t \le r \ \right\}, \\
\Lambda_n^{r+\frac{1}{2}} &=& \left\{ \  \mu = (\mu_1, \ldots, \mu_t) \vdash ( n-1) \ \big\vert \ \mu_2 + \cdots +  \mu_t \le r\ \right\}.
\end{array}
\end{equation}
There is an edge in $\mathfrak{B}_{n}$ from  $\lambda \in \Lambda_n^r$ to  $\mu \in \Lambda_n^{r+\frac{1}{2}}$ or $\mu \in \Lambda_n^{r-\frac{1}{2}}$  if and only if $\lambda = \mu + \square$.  For example, the Bratteli diagram $\mathfrak{B}_{6}$, for $0 \le r \le 3$, is shown in Figure \ref{fig:bratteli}.  The edges in $\mathfrak{B}_{n}$ describe the restriction and induction rules in \eqref{ResIndForSn} and \eqref{ResIndForGn}.  Since $\Lambda_n^0$ contains only the partition $(n)$, which labels the trivial $G_n$ or $S_n$-module, the Bratteli diagram has the property that the vertices on level $r$ label the irreducible $G_n$ modules which appear in  $\IR_q^r$ or, equivalently, the 
irreducible $S_n$ modules which appear in $\IR^r$.  Furthermore, the number of paths from the top of the diagram to $\lambda \in \Lambda_n^r$ is the multiplicity of $G_n^\lambda$ in $\IR_q^r$.   The number of these paths is also indicated below each vertex in Figure \ref{fig:bratteli}.

 \begin{figure} [htp]
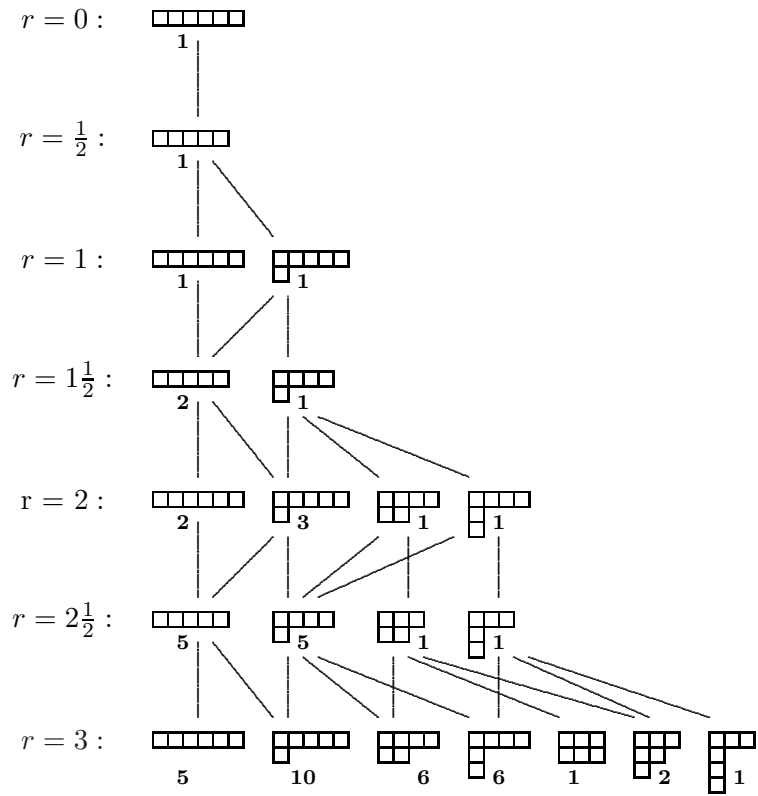

\label{fig:bratteli}
$$ 
{\beginpicture
\setcoordinatesystem units <0.2cm,0.2cm>         % sets scale
\setplotarea x from -6 to 40, y from 4 to 31   % sets plot size up
\linethickness=0.5pt                          % sets line thickness

%%%%%%%%%% level 3 %%%%%%%%%%%%%%%%%%%%%%%%%%%%%%%%%%%%%%%%%%%%
\put{$r = 3:$} at  -6 -.5

\color{black}
\put{$\scriptstyle{\bf 1}$} at 2 46
\put{$\scriptstyle{\bf 1}$} at 2 38
\put{$\scriptstyle{\bf 1}$} at 2 30  \put{$\scriptstyle{\bf 1}$} at 10 30
\put{$\scriptstyle{\bf 2}$} at 2 22 \put{$\scriptstyle{\bf 1}$} at 10 22
\put{$\scriptstyle{\bf 2}$} at 2 14 \put{$\scriptstyle{\bf 3}$} at 10 14   \put{$\scriptstyle{\bf 1}$} at 18 14  \put{$\scriptstyle{\bf 1}$} at 23 14
\put{$\scriptstyle{\bf 5}$} at 2 6  \put{$\scriptstyle{\bf 5}$} at 10 6  \put{$\scriptstyle{\bf 1}$} at 18 6  \put{$\scriptstyle{\bf 1}$} at 23 6
\put{$\scriptstyle{\bf 5}$} at 2 -3 \put{$\scriptstyle{\bf 10}$} at 10 -3  \put{$\scriptstyle{\bf 6}$} at 18 -3  \put{$\scriptstyle{\bf 6}$} at 23 -3
\put{$\scriptstyle{\bf 1}$} at 28 -3  \put{$\scriptstyle{\bf 2}$} at 34 -3  \put{$\scriptstyle{\bf 1}$} at 39 -3
\color{black}

\putrectangle corners at 0 -1 and 1 0
\putrectangle corners at 1 -1 and 2 0
\putrectangle corners at 2 -1 and 3 0
\putrectangle corners at 3 -1 and 4 0
\putrectangle corners at 4 -1 and 5 0
\putrectangle corners at 5 -1 and 6 0

\putrectangle corners at 8 -1 and 9 0
\putrectangle corners at 9 -1 and 10 0
\putrectangle corners at 10 -1 and 11 0
\putrectangle corners at 11 -1 and 12 0
\putrectangle corners at 12 -1 and 13 0
\putrectangle corners at 8 -2 and 9 -1

\putrectangle corners at 15 -1 and 16 0
\putrectangle corners at 16 -1 and 17 0
\putrectangle corners at 17 -1 and 18 0
\putrectangle corners at 18 -1 and 19 0
\putrectangle corners at 15 -2 and 16 -1
\putrectangle corners at 16 -2 and 17 -1

\putrectangle corners at 21 -1 and 22 0
\putrectangle corners at 22 -1 and 23 0
\putrectangle corners at 23 -1 and 24 0
\putrectangle corners at 24 -1 and 25 0
\putrectangle corners at 21 -2 and 22 -1
\putrectangle corners at 21 -3 and 22 -2

\putrectangle corners at 27 -1 and 28 0
\putrectangle corners at 28 -1 and 29 0
\putrectangle corners at 29 -1 and 30 0
\putrectangle corners at 27 -2 and 28 -1
\putrectangle corners at 28 -2 and 29 -1
\putrectangle corners at 29 -2 and 30 -1

\putrectangle corners at 32 -1 and 33 0
\putrectangle corners at 33 -1 and 34 0
\putrectangle corners at 34 -1 and 35 0
\putrectangle corners at 32 -2 and 33 -1
\putrectangle corners at 33 -2 and 34 -1
\putrectangle corners at 32 -3 and 33 -2

\putrectangle corners at 37 -1 and 38 0
\putrectangle corners at 38 -1 and 39 0
\putrectangle corners at 39 -1 and 40 0
\putrectangle corners at 37 -2 and 38 -1
\putrectangle corners at 37 -3 and 38 -2
\putrectangle corners at 37 -4 and 38 -3
%-------------------------------------------------

\plot 3 1 3 6  /
\plot 8 1 4 6  /
\plot 9 1 9 5  /
\plot 15 1 10 5  /
\plot 16 1 16 5  /
\plot 27 1 17 5  /
\plot 32 1 18 5  /
\plot 21 1 11 5  /
\plot 23 1 23 5  /
\plot 33 1 24 5  /
\plot 37 1 25 5  /

%%%%%%%%%% level 2+ %%%%%%%%%%%%%%%%%%%%%%%%%%%%%%%%%%%%%%%%%%%%
\put{$r = 2\frac{1}{2}:$} at  -6 7.5
\putrectangle corners at 0 7 and 1 8
\putrectangle corners at 1 7 and 2 8
\putrectangle corners at 2 7 and 3 8
\putrectangle corners at 3 7 and 4 8
\putrectangle corners at 4 7 and 5 8

\putrectangle corners at 8 7  and 9 8
\putrectangle corners at 9 7  and 10 8
\putrectangle corners at 10 7 and 11 8
\putrectangle corners at 11 7 and 12 8
\putrectangle corners at 8  6 and 9 7

\putrectangle corners at 15 7 and 16 8
\putrectangle corners at 16 7 and 17 8
\putrectangle corners at 17 7 and 18 8
\putrectangle corners at 15 6 and 16 7
\putrectangle corners at 16 6 and 17 7

\putrectangle corners at 21 7 and 22 8
\putrectangle corners at 22 7 and 23 8
\putrectangle corners at 23 7 and 24 8
\putrectangle corners at 21 6 and 22 7
\putrectangle corners at 21 5 and 22 6
%-------------------------------------------------

\plot 3 9 3 14  /
\plot 8 13 4 9  /
\plot 9 13 9 9  /
\plot 15 13 10 9  /
\plot 17 13 17 9  /
\plot 20 13 11 9  /
\plot 23 13 23 9  /

%%%%%%%%%% level 2 %%%%%%%%%%%%%%%%%%%%%%%%%%%%%%%%%%%%%%%%%%%%
\put{r = $2:$} at  -6 15.5
\putrectangle corners at 0  15 and 1  16
\putrectangle corners at 1  15 and 2  16
\putrectangle corners at 2  15 and 3  16
\putrectangle corners at 3  15 and 4  16
\putrectangle corners at 4  15 and 5  16
\putrectangle corners at 5  15 and 6  16

\putrectangle corners at 8  15 and 9   16
\putrectangle corners at 9  15 and 10  16
\putrectangle corners at 10 15 and 11  16
\putrectangle corners at 11 15 and 12  16
\putrectangle corners at 12 15 and 13  16
\putrectangle corners at 8  14 and 9  15

\putrectangle corners at 15  15 and 16  16
\putrectangle corners at 16  15 and 17  16
\putrectangle corners at 17  15 and 18  16
\putrectangle corners at 18  15 and 19  16
\putrectangle corners at 15  14 and 16  15
\putrectangle corners at 16  14 and 17  15

\putrectangle corners at 21  15 and 22  16
\putrectangle corners at 22  15 and 23  16
\putrectangle corners at 23  15 and 24  16
\putrectangle corners at 24  15 and 25  16
\putrectangle corners at 21  14 and 22  15
\putrectangle corners at 21  13 and 22 14
%-------------------------------------------------

\plot 3 17 3 22  /
\plot 8 17 4 22  /
\plot 9 17 9 21  /
\plot 15 17 10 21  /
\plot 21 17 11 21  /

%%%%%%%%%% level 1+ %%%%%%%%%%%%%%%%%%%%%%%%%%%%%%%%%%%%%%%%%%%%
\put{$r = 1\frac{1}{2}:$} at  -6 23.5
\putrectangle corners at 0  23 and 1  24
\putrectangle corners at 1  23 and 2  24
\putrectangle corners at 2  23 and 3  24
\putrectangle corners at 3  23 and 4  24
\putrectangle corners at 4  23 and 5  24

\putrectangle corners at 8  23 and 9   24
\putrectangle corners at 9  23 and 10  24
\putrectangle corners at 10 23 and 11  24
\putrectangle corners at 11 23 and 12  24
\putrectangle corners at 8  22 and 9  23
%-------------------------------------------------

\plot 3 25 3 30  /
\plot 8 29 4 25  /
\plot 9 29 9 25  /

%%%%%%%%%% level 1 %%%%%%%%%%%%%%%%%%%%%%%%%%%%%%%%%%%%%%%%%%%%
\put{$r = 1:$} at  -6 31.5
\putrectangle corners at 0  31 and 1  32
\putrectangle corners at 1  31 and 2  32
\putrectangle corners at 2  31 and 3  32
\putrectangle corners at 3  31 and 4  32
\putrectangle corners at 4  31 and 5  32
\putrectangle corners at 5  31 and 6  32

\putrectangle corners at 8  31 and 9   32
\putrectangle corners at 9  31 and 10  32
\putrectangle corners at 10 31 and 11  32
\putrectangle corners at 11 31 and 12  32
\putrectangle corners at 12 31 and 13  32
\putrectangle corners at 8  30 and 9  32

%%%%%%%%%% level 0+ %%%%%%%%%%%%%%%%%%%%%%%%%%%%%%%%%%%%%%%%%%%%
\put{$r = \frac{1}{2}:$} at  -6 39.5
\putrectangle corners at 0  39 and 1  40
\putrectangle corners at 1  39 and 2  40
\putrectangle corners at 2  39 and 3  40
\putrectangle corners at 3  39 and 4  40
\putrectangle corners at 4  39 and 5  40

%-------------------------------------------------

\plot 3 33 3 38  /
\plot 8 33 4 38  /

%%%%%%%%%% level 0 %%%%%%%%%%%%%%%%%%%%%%%%%%%%%%%%%%%%%%%%%%%%
\put{$r = 0:$} at  -6 47.5
\putrectangle corners at 0  47 and 1  48
\putrectangle corners at 1  47 and 2  48
\putrectangle corners at 2  47 and 3  48
\putrectangle corners at 3  47 and 4  48
\putrectangle corners at 4  47 and 5  48
\putrectangle corners at 5  47 and 6  48

%-------------------------------------------------

\plot 3 41 3 46  /

\endpicture}
$$
\caption{The Bratteli diagram $\mathfrak{B}_6$  for $0 \le r \le 3$.  The rows contain the partitions in $\Lambda_6^r$, the paths from the top of the diagram to $\lambda \in \Lambda_6^r$ are the $r$-vacillating tableaux of shape $\lambda$, and the number of paths to $\lambda$ is $m_r^\lambda$.   The label on vertex $\lambda$ is  $m_r^\lambda$.}
\end{figure}

From double-centralizer theory (see for example \cite[Theorem 5.4]{HR})
we have the following properties which justify calling $Q_r(n,q)$ a $q$-analog of the partition algebra $P_r(n)$. These results follow from the fact that 
the decomposition of $\IR^r_q$  as a $G_n$ module is the same as the decomposition of 
$\IR^r_1$  as an $S_n$-module.  We assume $r,n \in \ZZ_{\ge 0}$ and $n \ge 2r$,
\begin{enumerate} 
\item The irreducible representations of $P_r(n)$ and $Q_r(n,q)$ are each indexed by the partitions in $\Lambda_n^r$. These are the partitions which label the irreducible $S_n$ and $G_n$ modules which appear in $\IR_1^r$ and $\IR_q^r$, respectively. 

\item As bimodules for $(S_n, P_r(n))$ and $(G_n, Q_r(n,q))$,  we have
\begin{equation}\label{intro:decomp}
\IR^r_1 \cong \bigoplus_{\lambda \in \Lambda_n^r} S_n^\lambda \otimes P_r^\lambda \quad\hbox{ and }\quad
\IR_q^r \cong \bigoplus_{\lambda \in \Lambda_n^r}  G_n^\lambda \otimes Q_r^\lambda,
\end{equation}
where $P_r^\lambda$ and $Q_r^\lambda$ are irreducible $P_r(n)$ and $Q_r(n,q)$-modules, respectively.

\item $m_r^\lambda = \dim(P_r^\lambda) = \dim(Q_r^\lambda)$ equals the multiplicity of $S_n^\lambda$ in $\IR^r$ (or, equivalently, the multiplicity of $G_n^\lambda$ in $\IR_q^r$).  The value of $m_r^\lambda$ can be computed by counting paths to $\lambda$ in the Bratteli diagram.  Each of these paths is a sequence of partitions obtained by iteratively removing and adding a box; these are called $r$-vacillating tableaux of shape $\lambda$ (see Section 2) and they are also studied in \cite{CDDSY} and \cite{HL}.

\item By general Wedderburn theory,  $\dim(P_r(n)) = \sum_{\lambda \in \Lambda_n^r} \dim(P_r^\lambda)^2 =  \sum_{\lambda \in \Lambda_n^r} \dim(Q_r^\lambda)^2 = \dim(Q_r(n,q))$.  It follows that $\dim(P_r(n)) = \dim(Q_r(n,q))$ equals the number of set partitions of $\{1, 2, \ldots, 2r\}$ which is the Bell number $B(2r)$.

\item For $r \in \ZZ_{>0}$, there is a natural chain of embeddings 
$$
Q_{r-\frac{1}{2}}(n,q) \subseteq Q_r(n,q)  \subseteq Q_{r + \frac{1}{2}}(n,q),
$$
The restriction rule for $Q_{r-\frac{1}{2}}(n,q) \subseteq Q_r(n,q)$ is given by the Bratteli diagram. Namely, the module $Q^\mu_{r - \frac{1}{2}}$ is a component of $\Res^{Q_r(n,q)}_{Q_{r-\frac{1}{2}}(n,q)}  (Q^\lambda_r)$ if and only if $\mu$ and $\lambda$ are connected by an edge in $\mathfrak{B}_n$.  This same rule holds if $Q_r(n,q)$ is replaced by $P_r(n)$.
\end{enumerate}

Since restriction from $S_n$ to $S_{n-1}$ preserves the dimension of an $S_n$-module, and induction from $S_{n-1}$ to $S_n$ increases the dimension by a factor of $n$, we have that $\dim(\IR^r_1) = n^r$. This is to be expected, since $\IR^r_1 \cong V^{\otimes r}$ and $\dim(V) = n$. 
The dimension of the irreducible symmetric group modules are $\dim(S^\lambda_n) = f^\lambda_n$, the number of standard Young tableaux of shape $\lambda$ (given by the hook formula). The dimensions of the irreducible partition algebra modules are $\dim(P_r^\lambda) = m_r^\lambda$,  the number of $r$-vacillating tableau of shape $\lambda$.  Computing dimensions on both sides of the first equation in \eqref{intro:decomp} gives the identity
\begin{equation}\label{intro:Pnidentity}
n^r = \sum_{\lambda\vdash n} f^\lambda_n m^\lambda_r.
\end{equation}
A combinatorial proof of  \eqref{intro:Pnidentity} is given in  \cite{HL}  by defining a Schensted-like insertion/deletion process to find a bijection 
$$
\left\{ (a_1, \ldots, a_r ) \right\}  \longleftrightarrow  \left\{ (P,Q) \right\}
$$
between integer sequences $(a_1, \ldots, a_r)$ with $a_j \in \{ 1, \ldots, n \}$ and pairs $(P,Q),$ where $P$ is a standard Young tableau of shape $\lambda \vdash n$ and $Q$ is an $r$-vacillating tableau  of shape $\lambda \vdash n$.    If we compute the dimensions on each side of the second equation in  \eqref{intro:decomp} we get
\begin{equation}
d_{n,r}(q) = \dim(\IR_q^r) = \sum_\lambda f^\lambda_n(q) m^\lambda_k,
\end{equation}
where now $f^\lambda_n(q)  = \dim(G_n^\lambda)$ is the well-known $q$-analog of $f^\lambda_n$, given by the q-hook formula  and $d_{n,r}(q)$ is a $q$-polynomial which satisfies $d_{n,r}(1) = n^r$. 
In Section 2, we compute this dimension explicitly by carrying the $q$-weights from $f^\lambda(q)$ across  the Schensted insertion.

\end{section}

\begin{section}{Combinatorial Computation of $d_{n,r}(q)$}

This section gives a purely combinatorial derivation of the formula for the $q$-polynomial $d_{n,r}(q)$.   In Section 3, we give a basis of $\IR_q^r$ and we count the elements of the basis to give another proof that $\dim(\IR_q^r) = d_{n,r}(q)$.

\begin{subsection}{The Delete-Insert Schensted Algorithm}

For $n,r \in \ZZ_{>0},$ define
$$
\{1, \ldots,n\}^r= \left\{\  (a_1, \ldots, a_r) \ \big\vert \ a_j \in \{1,\ldots, n\} \ \right\}.
$$
This set has cardinality $n^r$.
For a partition $\lambda \vdash n$, a standard tableau of shape $\lambda$ is a filling of the boxes of the Young diagram of $\lambda$ with integers $1, \ldots, n$ such that the rows increase left-to-right and the columns increase top-to-bottom.  As in \cite{HL} we define an algorithm that maps sequences in $\{1,\ldots,n\}^r$ to standard tableaux. 
Let $a = (a_1, a_2 \ldots, a_r)$ and recursively define $P_i$ and $P_{i+\frac{1}{2}}$ for $0\le i \le r$, by 
\begin{equation}\label{InsertionDef}
\begin{array}{rcll}
P_0 &=& {\beginpicture
\setcoordinatesystem units <0.4cm,0.4cm>        
\setplotarea x from 1 to 7.3, y from -1 to 1   
\linethickness=0.5pt   
\putrectangle corners at 1  -.5 and  2 .5
\putrectangle corners at 2  -.5 and  3 .5
\putrectangle corners at 3  -.5 and  6 .5
\putrectangle corners at 6  -.5 and  7 .5
\put{$\scriptstyle{1}$} at 1.5 0
\put{$\scriptstyle{2}$} at 2.5 0
\put{$\cdots$} at 4.5 0
\put{$\scriptstyle{n}$} at 6.5 0
\endpicture}, \\
P_{i+\frac{1}{2}} &=& P_{i-1}  \jdt a_i , \qquad & 0 \le i \le r-1, \\
P_{i+1} &=& P_{i+\frac{1}{2}} \RSK a_i, \qquad & 0 \le i \le r-1,
\end{array}
\end{equation}
where this notation means that we first remove the letter $a_i$ from $P_{i}$ using Sch\"utzenberger's {\it jeu-de-taquin} to get a tableau $P_{i+\frac{1}{2}}$, and then we reinsert $a_i$ into $P_{i+\frac{1}{2}}$ using \emph{Robinson--Schensted--Knuth row  insertion} to obtain $P_{i+1}$.  See \cite[A1.2,7.11]{Sta2} for the definitions of jeu-de-taquin and RSK insertion.
Example \ref{ex:DI} provides an example of the application of this algorithm.

For $0 \le i \le k$, let $\lambda^{(i)}$ be the partition shape of the tableau $P_i$ and let $\lambda^{(i+\frac{1}{2})}$ be the partition shape of $P_{i+\frac{1}{2}}$.   The final tableau $P_a= P_r$ that results from the insertion of $a = (a_1, \ldots, a_r)$ is the \emph{insertion tableau}. It is a standard Young tableaux of shape $\lambda = \lambda^{(r)}$.  The sequence of shapes that arise along the way,
$$
Q_a = \left((n) =  \lambda^{(0)}, \lambda^{(\frac{1}{2})}, \lambda^{(1)}, \ldots, \lambda^{(r)} = \lambda \right),
$$
is the \emph{recording tableau} of the sequence.   The recording tableaux that appear in this process are uniquely described by the following properties:
\begin{enumerate}
\item $\lambda^{(0)} = (n),$
\item For  $0 \le i \le r-1$,  $\lambda^{(i+\frac{1}{2})}$ is a partition of $n-1$ that is obtained from  $\lambda^{(i)}$ by deleting a box,
\item For  $1 \le i \le r$,  $\lambda^{(i)}$ is a partition of $n$ that is obtained from $\lambda^{(i-\frac{1}{2})}$ by adding a box.
\end{enumerate}
If $\lambda = \lambda^{(r)}$, then a tableau satisfying these properties is called a $r$-\emph{vacillating tableau of shape $\lambda$}.  See \cite{HL} and \cite{CDDSY}.  The partition shapes that appear in the $\ell$th step in the process of inserting $a \in \{1,\ldots n\}^r$ are in the set
\begin{equation*}
\begin{array}{rcl}
\Lambda_n^\ell& =& \left\{ \  \lambda = (\lambda_1, \ldots, \lambda_t) \vdash  n \ \big\vert \  \lambda_2 + \cdots +  \lambda_t \le \ell \ \right\}. \\
\Lambda_n^{\ell+\frac{1}{2}} &=& \left\{ \  \lambda = (\lambda_1, \ldots, \lambda_t) \vdash  n-1 \ \big\vert \  \lambda_2 + \cdots +  \lambda_t \le \ell\ \right\}.
\end{array}
\end{equation*}

The $r$-vacillating tableau also appear in the Bratteli diagram $\mathfrak{B}_n$ shown in Figure \ref{fig:bratteli} for  
$n = 6$ and $0 \le r \le 3$.  The paths from the top of the diagram to  $\lambda$ on level $r$ are the $r$-vacillating tableaux of shape $\lambda$, and $m_r^\lambda$ is the number of $r$-vacillating tableaux of shape $\lambda$.  When $r \ge 2n$, the number is independent of $n$.  We refer to these paths as ``tableaux" since they determine paths in the Bratteli diagram in the same way that standard Young tableaux determine paths in Young's lattice.

We let $a \DI (P_a, Q_a)$ denote the ``delete-insert" process defined in \eqref{InsertionDef}, which associates each $a\in \{1,\ldots, n\}^r$ with a pair $(P_a,Q_a)$ consisting of a standard tableau $P_a$ and an $r$-vacillating tableaux $Q_a$, each of shape $\lambda \in \Lambda_n^r$.  In \cite{HL} this algorithm is shown to be invertible and thus provides a bijection
\begin{equation}\label{bijection}
\{1, \ldots ,n\}^r \DI \bigsqcup_{\lambda \in \Lambda_n^r} \left\{ \ (P,Q) \ \Big\vert 
\begin{array}{l}
\hbox{ $P$ is a standard tableau of shape $\lambda$}\\
\hbox{ $Q$ is a $r$-vacillating tableau of shape $\lambda$}\\
\end{array}
\right\}.
\end{equation}
This gives a combinatorial proof of the identity
\begin{equation}\label{PartAlgIdentity}
n^r =  \sum_{\lambda \in \Lambda^r_n} f^\lambda_n m^\lambda_r,
\end{equation}
where $f^\lambda$ is the number of standard tableaux of shape $\lambda$ (given by the hook formula), and 
$m^\lambda_r$ is the number of $r$-vacillating tableaux of shape $\lambda$.

\end{subsection}

\begin{subsection}{Delete/Insertion and Major Index}

We now show that the bijection \eqref{InsertionDef} carries the backsteps associated to integer sequences  to the descent set on standard tableaux.  We first map sequences in $\{1,\ldots,n\}^r$ to permutations in $S_n$ using  following surjection
\begin{equation}\label{wadef}
\begin{array}{ccl}
 \{1,\ldots, n\}^r & \to & S_n \\
a= (a_1, \ldots, a_r) & \mapsto & w_a = \rightmost( 1,2, \ldots, n, a_1, \ldots, a_r )
\end{array}
\end{equation}
where $\rightmost( 1,2, \ldots, n, a_1, \ldots, a_r )$ is the permutation consisting of the rightmost occurrence of each integer in $\{1, \ldots, n\}$. For example, 
$$
a = (2,1,3,1,6,4,6,3,4)\quad  \mapsto \quad w_a = \rightmost( 1,2, 3,4,5,6,2,1,3,1,6,4,6,3,4 ) = (5,2,1,6,3,4).
$$
Alternatively, we can produce $w_a = (b_1,b_2, \ldots, b_n)$ iteratively using the following algorithm. 
\begin{equation}\label{iterativepi}
\begin{array}{lll}
(1)& \begin{array}{l}\hbox{$w^{(0)}  =  (1,2, \ldots, n)$,}\end{array} & \\
(2)&
\begin{array}{l}
\hbox{$w^{(i+1)}$ is obtained from $w^{(i)}$ by deleting $a_i$ from $w^{(i)}$} \\
\hbox{and then appending $a_i$ to the right  of $w^{(i)}$,}
\end{array}  &\quad 1 \le i < r. \\
(3)&\begin{array}{l} \hbox{$w_a  = w^{(r)}$.} \end{array} \\
\end{array}
\end{equation}
Applying this algorithm to $a = (2,1,3,1,6,4,6,3,4)$, for example,  yields the same $w_a$ as above:
$$
\begin{array}{llll}
w^{(0)} = (1,2,3,4,5,6)&&\qquad &w^{(5)} = (4,5,2,3,1,6) \\
w^{(1)} = (1,3,4,5,6,2)&&&w^{(6)} = (5,2,3,1,6,4) \\
w^{(2)} = (3,4,5,6,2,1)&&&w^{(7)} = (5,2,1,6,4,3) \\
w^{(3)} = (4,5,6,2,1,3)&&& w^{(8)} = (5,2,1,6,3,4)\\
w^{(4)} = (4,5,6,2,3,1)&&& w_a =  (5,2,1,6,3,4).\\
\end{array}
$$
It is clear that the processes defined in \eqref{wadef} and \eqref{iterativepi} yield the same result since the elements of $a$ are cycled to the right end of of $w_a$ in the order that they appear in $a$..

The \emph{backsteps} (see for example \cite{Lo}) in a permutation $w = (w_1, w_2, \ldots, w_n) \in S_n$ are
\begin{equation}
\BS(w)  = \left\{\ i \ \big\vert\ i+1 \hbox{ is to the left of $i$ in $w=(w_1, w_2, \ldots, w_n)$} \ \right\}.
\end{equation}
The \emph{descent set} in $w \in S_n$ is defined by  $\Des(w) =  \Des(w_1, w_2, \ldots, w_n) = \{\ i \ \vert\  w_i > w_{i+1} \ \},$
and it is easy to check that $\BS(w) = \Des(w^{-1})$.   For example if $w = (5,2,1,6,3,4)$ then $\BS(w)  = \Des(w^{-1}) = \{1, 4\}$.
If $P$ is a standard tableau, then the {descent set} of $P$ is
\begin{equation}
\Des(P) =  \left\{\ i \ \vert\  \hbox{ $i+1$ is in a lower row than $i$  in $P$}\ \right\}.
\end{equation}
For example $\Des\left(
{\beginpicture
\setcoordinatesystem units <0.4cm,0.4cm>        
\setplotarea x from 1 to 5, y from -1 to 1   
\linethickness=0.5pt   
\putrectangle corners at 1  .5 and  2 1.5
\putrectangle corners at 2  .5 and  3 1.5
\putrectangle corners at 3  .5 and  4 1.5
\putrectangle corners at 4  .5 and  5 1.5
\putrectangle corners at 1  -.5 and  2 .5
\putrectangle corners at 2  -.5 and  3 .5
\putrectangle corners at 3  -.5 and  4 .5
\putrectangle corners at 4  -.5 and  5 .5
\putrectangle corners at 1  -1.5 and  2 -.5
\putrectangle corners at 2  -1.5 and  3 -.5
\put{$\scriptstyle{1}$} at 1.5 1
\put{$\scriptstyle{2}$} at 2.5 1
\put{$\scriptstyle{5}$} at 3.5 1
\put{$\scriptstyle{6}$} at 4.5 1
\put{$\scriptstyle{3}$} at 1.5 0
\put{$\scriptstyle{7}$} at 2.5 0
\put{$\scriptstyle{9}$} at 3.5 0
\put{$\scriptstyle{10}$} at 4.5 0
\put{$\scriptstyle{4}$} at 1.5 -1
\put{$\scriptstyle{8}$} at 2.5 -1
\endpicture}\right) = \{2,3, 6,7\}.$
See Example \ref{ex:DI} for an illustration of the following proposition.

\begin{prop} If $a \in \{1,\ldots,n\}^r$ and $a \DI (P_a,Q_a)$, where $P_a$ is a standard tableau of shape $\lambda \in \Lambda_n^r$ and $Q_a$ is an $r$-vacillating tableau, then 
$$
\BS( w_a ) = \Des(P_a).
$$
\end{prop}

\begin{proof}  The proof is by induction on the length $r$ of $a = (a_1, \ldots, a_r)$.  If $r = 0$, then $w = \emptyset$ and $P_a = {\beginpicture
\setcoordinatesystem units <0.4cm,0.4cm>        
\setplotarea x from 1 to 7.3, y from -1 to 1   
\linethickness=0.5pt   
\putrectangle corners at 1  -.5 and  2 .5
\putrectangle corners at 2  -.5 and  3 .5
\putrectangle corners at 3  -.5 and  6 .5
\putrectangle corners at 6  -.5 and  7 .5
\put{$\scriptstyle{1}$} at 1.5 0
\put{$\scriptstyle{2}$} at 2.5 0
\put{$\cdots$} at 4.5 0
\put{$\scriptstyle{n}$} at 6.5 0
\endpicture}$.  In this case, $w_a = (1,2,\ldots,n)$ has $\BS(w_a) = \emptyset = \Des(P_a)$.

Now let $r >0$ and  $(a_1, \ldots, a_{r-1}) \DI  (P_{r-1},Q_{r-1}).$  Then $P_a = (P_{r-1} \jdt a_r) \RSK a_r$, and 
 by induction $\Des(P_{r-1}) = \BS( w_{(a_1, \ldots, a_{r-1})}).$
 By \eqref{iterativepi}, the permutation $w_{(a_1, \ldots, a_r)}$ is the same as  $w_{(a_1, \ldots, a_{r-1})}$ except that it has $a_r$ moved to the the rightmost position.   Since $a_r$ is now to the right of both $a_r-1$ and $a_r+1$, and this is the only changed made, we know that 
\begin{enumerate}
\item[(a)] $(a_r-1,  a_r)$ is  not a backstep  in $w_a$,
\item[(b)] $(a_r, a_r+1)$ is a backstep  in $w_a$, and
\item[(c)]  all other $(i,i+1)$ relationships are the same in $w_a$ as they were in $w_{(a_1, \ldots, a_{r-1})}$.
\end{enumerate}
These same relationships happen in $P$:
\begin{enumerate}
\item[(a')]  When $a_r$ is deleted from $P_{r-1}$ (via jeu-de-taquin) and then reinserted (via RSK), it ends up in the first row of $P$.  Thus $(a_{r}-1,a_r)$ is not a descent in $P$. 
\item[(b')]  If $a_{r}+1$ was in the first row of $P_{r-1}$ then $a_r$ bumps it to a lower row. Otherwise, it was already in a lower row, and either way $(a_{r},a_r+1)$ is a descent in $P$. 
\item[(c')]  Whenever $i$ gets bumped into the next row, if $i+1$ is in that row, $i$ will bump $i+1$ into a lower row.  So if $(i,i+1)$ is a descent it will remain a descent. If $(i,i+1)$ is not a descent, then we must consider the case when $i+1$ gets bumped lower than $i$. This only happens if $i$ and $i+1$ are in the same row.  But in this case a number that might bump $i+1$ would have to be lower than $i+1$ and thus lower than $i$. So it might potentially bump $i$ but it would not bump $i+1$.
\end{enumerate}
It follows by induction that $\Des(P_a) = \BS(w_a)$, as desired.  \end{proof}

The {inverse major index} $\imaj$ of a permutation $w \in S_n$ is the sum of the backsteps in $w$, and the major index $\maj$ of a standard tableau $P$ is the sum of the descents in $P$. That is,
\begin{equation}\label{def:majorindex}
\imaj(w) = \sum_{i \in \BS(w)}\! i \qquad \hbox{ and } \qquad  \maj(P) = \sum_{i \in \Des(P)}\!  i.
\end{equation}
Note that the major index of $w$ is  $\maj(w) = \sum_{i \in \Des(w)} i$, and $\imaj(w) = \maj(w^{-1})$. 
Let $q$ be an indeterminate (in Section 3 we will specialize $q$ to be a prime power).   For $\lambda \vdash n$, a $q$-analog of the hook number $f^\lambda_n$ is given by
\begin{equation}\label{qhook}
f^\lambda_n(q) = \sum_T q^{\maj(T)},
\end{equation}
where the sum is over all standard tableaux $T$ of shape $\lambda$.  Then $f^\lambda_n(q)$ is the dimension of the irreducible unipotent $\GL_n(\FF_q)$-module labeled by $\lambda$ and it is also given by the $q$-hook formula (see \cite[IV.6.7]{Mac}).

\begin{cor}\label{corsum}  For all $n,r \in \ZZ_{> 0}$, we have
$$
\sum_{a \in\{1,\ldots,n\}^r } q^{{\imaj}(w_a)} = 
\sum_{\lambda\in \Lambda_n^r} \sum_{(P,Q)} q^{{\maj}( P)} 
=  \sum_{\lambda\in \Lambda_n^r} f^\lambda_n(q) m^\lambda_r,
$$
where $(P,Q)$ ranges over all pairs consisting of a standard tableau $P$ of shape $\lambda$ and an $r$-vacillating tableau $Q$ of shape $\lambda$, and $w_a$ is defined in \eqref{wadef}.
\end{cor}

\begin{proof} The first equality follows immediately from the fact that the delete-insert bijection \eqref{bijection} pairs $a \in \{1,\ldots n\}^r$ with $\{ (P_a,Q_a)\}$ and carries $q^{{\imaj}( w_a)}$ to $q^{{\maj}( P)}$.  The second equality follows from \eqref{qhook} and from the fact that $m_r^\lambda$ equals the number of $r$-vacillating tableaux $Q$ of shape $\lambda$.
\end{proof}

\begin{example} \label{ex:DI} The following table illustrates the process of delete-inserting the sequence $a = (3,5,2,3,2) \in \{1, \ldots, 6\}^5$ to produce a pair $(P_a,Q_a)$ of shape $\lambda = (2,2,1,1).$ The reader should observe that at each step in  this process the backsteps in $w_a$ equal the descents in $P_a$.
$$
\begin{tabular}{|c|c|cc|c|c|c|}
\hline
$i$ & $a_i$ & $P_a$ & & $a$ & $w_a$ & $\BS(w_a)$ = $\Des(P_a)$ \\
\hline \hline
$0$ &  & 
{\beginpicture
\setcoordinatesystem units <0.4cm,0.4cm>        
\setplotarea x from 1 to 7.3, y from -1 to 1   
\linethickness=0.5pt   
\putrectangle corners at 1  -.5 and  2 .5
\putrectangle corners at 2  -.5 and  3 .5
\putrectangle corners at 3  -.5 and  4 .5
\putrectangle corners at 4  -.5 and  5 .5
\putrectangle corners at 5  -.5 and  6 .5
\putrectangle corners at 6  -.5 and  7 .5
\put{$\scriptstyle{1}$} at 1.5 0
\put{$\scriptstyle{2}$} at 2.5 0
\put{$\scriptstyle{3}$} at 3.5 0
\put{$\scriptstyle{4}$} at 4.5 0
\put{$\scriptstyle{5}$} at 5.5 0
\put{$\scriptstyle{6}$} at 6.5 0
\endpicture} & & $\emptyset$& $(1,2,3,4,5,6)$ & $\emptyset$  \\
\hline
$\frac{1}{2}$ & $3$ & 
{\beginpicture
\setcoordinatesystem units <0.4cm,0.4cm>        
\setplotarea x from 1 to 7.3, y from -1 to 1   
\linethickness=0.5pt   
\putrectangle corners at 1  -.5 and  2 .5
\putrectangle corners at 2  -.5 and  3 .5
\putrectangle corners at 3  -.5 and  4 .5
\putrectangle corners at 4  -.5 and  5 .5
\putrectangle corners at 5  -.5 and  6 .5
\put{$\scriptstyle{1}$} at 1.5 0
\put{$\scriptstyle{2}$} at 2.5 0
\put{$\scriptstyle{4}$} at 3.5 0
\put{$\scriptstyle{5}$} at 4.5 0
\put{$\scriptstyle{6}$} at 5.5 0
\endpicture} & $\jdt 3$  & &  &  \\
%\hline
$1$ &  & 
{\beginpicture
\setcoordinatesystem units <0.4cm,0.4cm>        
\setplotarea x from 1 to 7.3, y from -1 to 1   
\linethickness=0.5pt   
\putrectangle corners at 1  -.5 and  2 .5
\putrectangle corners at 2  -.5 and  3 .5
\putrectangle corners at 3  -.5 and  4 .5
\putrectangle corners at 4  -.5 and  5 .5
\putrectangle corners at 5  -.5 and  6 .5
\putrectangle corners at 1  -1.5 and  2 -.5
\put{$\scriptstyle{1}$} at 1.5 0
\put{$\scriptstyle{2}$} at 2.5 0
\put{$\scriptstyle{3}$} at 3.5 0
\put{$\scriptstyle{5}$} at 4.5 0
\put{$\scriptstyle{6}$} at 5.5 0
\put{$\scriptstyle{4}$} at 1.5 -1
\endpicture} & $\RSK 3$ & $(3)$ & $(1,2,4,5,6,3)$ & $\{3\}$   \\
\hline
$1\frac{1}{2}$ & $5$ & 
{\beginpicture
\setcoordinatesystem units <0.4cm,0.4cm>        
\setplotarea x from 1 to 7.3, y from -1 to 1   
\linethickness=0.5pt   
\putrectangle corners at 1  -.5 and  2 .5
\putrectangle corners at 2  -.5 and  3 .5
\putrectangle corners at 3  -.5 and  4 .5
\putrectangle corners at 4  -.5 and  5 .5
\putrectangle corners at 1  -1.5 and  2 -.5
\put{$\scriptstyle{1}$} at 1.5 0
\put{$\scriptstyle{2}$} at 2.5 0
\put{$\scriptstyle{3}$} at 3.5 0
\put{$\scriptstyle{6}$} at 4.5 0
\put{$\scriptstyle{4}$} at 1.5 -1
\endpicture} & $\jdt 5$  &  &  & \\
%\hline
$2$ & & 
{\beginpicture
\setcoordinatesystem units <0.4cm,0.4cm>        
\setplotarea x from 1 to 7.3, y from -1 to 1   
\linethickness=0.5pt   
\putrectangle corners at 1  -.5 and  2 .5
\putrectangle corners at 2  -.5 and  3 .5
\putrectangle corners at 3  -.5 and  4 .5
\putrectangle corners at 4  -.5 and  5 .5
\putrectangle corners at 1  -1.5 and  2 -.5
\putrectangle corners at 2  -1.5 and  3 -.5
\put{$\scriptstyle{1}$} at 1.5 0
\put{$\scriptstyle{2}$} at 2.5 0
\put{$\scriptstyle{3}$} at 3.5 0
\put{$\scriptstyle{5}$} at 4.5 0
\put{$\scriptstyle{4}$} at 1.5 -1
\put{$\scriptstyle{6}$} at 2.5 -1
\endpicture} 
&  $\RSK 5$ & $(3,5)$ & $(1,2,4,6,3,5)$ & $\{3,5\}$   \\
\hline
$2\frac{1}{2}$ &  $2$  & 
{\beginpicture
\setcoordinatesystem units <0.4cm,0.4cm>        
\setplotarea x from 1 to 7.3, y from -1 to 1   
\linethickness=0.5pt   
\putrectangle corners at 1  -.5 and  2 .5
\putrectangle corners at 2  -.5 and  3 .5
\putrectangle corners at 3  -.5 and  4 .5
\putrectangle corners at 1  -1.5 and  2 -.5
\putrectangle corners at 2  -1.5 and  3 -.5
\put{$\scriptstyle{1}$} at 1.5 0
\put{$\scriptstyle{3}$} at 2.5 0
\put{$\scriptstyle{5}$} at 3.5 0
\put{$\scriptstyle{4}$} at 1.5 -1
\put{$\scriptstyle{6}$} at 2.5 -1
\endpicture}
& $\jdt 2$  &  &  & \\
%\hline
$3$ &  & 
{\beginpicture
\setcoordinatesystem units <0.4cm,0.4cm>        
\setplotarea x from 1 to 7.3, y from -1 to 1   
\linethickness=0.5pt   
\putrectangle corners at 1  -.5 and  2 .5
\putrectangle corners at 2  -.5 and  3 .5
\putrectangle corners at 3  -.5 and  4 .5
\putrectangle corners at 1  -1.5 and  2 -.5
\putrectangle corners at 2  -1.5 and  3 -.5
\putrectangle corners at 1  -2.5 and  2 -1.5
\put{$\scriptstyle{1}$} at 1.5 0
\put{$\scriptstyle{2}$} at 2.5 0
\put{$\scriptstyle{5}$} at 3.5 0
\put{$\scriptstyle{3}$} at 1.5 -1
\put{$\scriptstyle{6}$} at 2.5 -1
\put{$\scriptstyle{4}$} at 1.5 -2
\endpicture}
& $\RSK 2$ & $(3,5,2)$ & $(1,4,6,3,5,2)$ & $\{2,3,5\}$  \\
\hline
$3\frac{1}{2}$ & $3$ & 
{\beginpicture
\setcoordinatesystem units <0.4cm,0.4cm>        
\setplotarea x from 1 to 7.3, y from -1 to 1   
\linethickness=0.5pt   
\putrectangle corners at 1  -.5 and  2 .5
\putrectangle corners at 2  -.5 and  3 .5
\putrectangle corners at 3  -.5 and  4 .5
\putrectangle corners at 1  -1.5 and  2 -.5
\putrectangle corners at 2  -1.5 and  3 -.5
\put{$\scriptstyle{1}$} at 1.5 0
\put{$\scriptstyle{2}$} at 2.5 0
\put{$\scriptstyle{5}$} at 3.5 0
\put{$\scriptstyle{4}$} at 1.5 -1
\put{$\scriptstyle{6}$} at 2.5 -1
\endpicture}
& $\jdt 3$ &  &  &  \\
%\hline
$4$ &  & 
{\beginpicture
\setcoordinatesystem units <0.4cm,0.4cm>        
\setplotarea x from 1 to 7.3, y from -1 to 1   
\linethickness=0.5pt   
\putrectangle corners at 1  -.5 and  2 .5
\putrectangle corners at 2  -.5 and  3 .5
\putrectangle corners at 3  -.5 and  4 .5
\putrectangle corners at 1  -1.5 and  2 -.5
\putrectangle corners at 2  -1.5 and  3 -.5
\putrectangle corners at 1  -2.5 and  2 -1.5
\put{$\scriptstyle{1}$} at 1.5 0
\put{$\scriptstyle{2}$} at 2.5 0
\put{$\scriptstyle{3}$} at 3.5 0
\put{$\scriptstyle{4}$} at 1.5 -1
\put{$\scriptstyle{5}$} at 2.5 -1
\put{$\scriptstyle{6}$} at 1.5 -2
\endpicture}
& $\RSK 3$ & $(3,5,2,3)$ & $(1,4,6,5,2,3)$ & $\{3,5\}$   \\
\hline
$4\frac{1}{2}$ & $2$  & 
{\beginpicture
\setcoordinatesystem units <0.4cm,0.4cm>        
\setplotarea x from 1 to 7.3, y from -1 to 1   
\linethickness=0.5pt   
\putrectangle corners at 1  -.5 and  2 .5
\putrectangle corners at 2  -.5 and  3 .5
\putrectangle corners at 1  -1.5 and  2 -.5
\putrectangle corners at 2  -1.5 and  3 -.5
\putrectangle corners at 1  -2.5 and  2 -1.5
\put{$\scriptstyle{1}$} at 1.5 0
\put{$\scriptstyle{3}$} at 2.5 0
\put{$\scriptstyle{4}$} at 1.5 -1
\put{$\scriptstyle{5}$} at 2.5 -1
\put{$\scriptstyle{6}$} at 1.5 -2
\endpicture}
& $\jdt 2$&  &  &  \\
%\hline
$5$ &  & 
{\beginpicture
\setcoordinatesystem units <0.4cm,0.4cm>        
\setplotarea x from 1 to 7.3, y from -1 to 1   
\linethickness=0.5pt   
\putrectangle corners at 1  -.5 and  2 .5
\putrectangle corners at 2  -.5 and  3 .5
\putrectangle corners at 1  -1.5 and  2 -.5
\putrectangle corners at 2  -1.5 and  3 -.5
\putrectangle corners at 1  -2.5 and  2 -1.5
\putrectangle corners at 1  -3.5 and  2 -2.5
\put{$\scriptstyle{1}$} at 1.5 0
\put{$\scriptstyle{2}$} at 2.5 0
\put{$\scriptstyle{3}$} at 1.5 -1
\put{$\scriptstyle{5}$} at 2.5 -1
\put{$\scriptstyle{4}$} at 1.5 -2
\put{$\scriptstyle{6}$} at 1.5 -3
\endpicture}
&$\RSK 2$ & $(3,5,2,3,2)$  & $(1,4,6,5,3,2)$ & $\{2,3,5\}$   \\
\hline
\end{tabular}
$$
\end{example}

\end{subsection}

\newpage

\begin{subsection}{Set Partitions and Major Index}

For an integer $i \ge 0$, define
\begin{equation}
[i] = \frac{q^i-1}{q-1} = q^{i-1} + q^{i-2} + \cdots + 1,
\end{equation}
so that $[i]_{q=1} = i$.  Recall that the \emph{Stirling number} $S(r, \ell)$ is the number of set partitions of a set of size $r$ into $\ell$ subsets.  We now compute the sum that appears in Corollary \ref{corsum}.

\begin{prop} \label{propsum} For $r,n \in \ZZ_{>0},$
$$
\displaystyle{ \sum_{a\in \{ 1,\ldots, n\}^r} q^{{\imaj}(w_a)} =\sum_{\ell = 1}^n  S(r,\ell)  [n] [n-1] \cdots [n-\ell+1]}.
$$

\end{prop}

\begin{proof} We begin by classifying the permutations $w_a$ that appear in the sum.  For each sequence $a = (a_1, \ldots, a_r)\in \{1,\ldots,n\}^r$ we define $\shape(a)$ to be the set partition of $\{1, \ldots, r\}$ given by the rule,
$$
\hbox{$i \sim j$ in $\shape(a)$ \quad if and only if \quad $a_i = a_j$ in $a$}.
$$
We also let
$$
D_{n,t} = \{\ w = (w_1, \ldots, w_n) \in S_n \ \vert \ w_1 < w_2 < \ldots < w_{t} \ \},
$$
be a distinct set of minimal-length coset representatives of $S_n/S_t$,  where we naturally embed $S_t \subseteq S_n$ as the permutations of $\{1, \ldots, t\}$. From this construction, we immediately have,
$$
\hbox{if \quad $\shape(a)$ has $\ell$ parts \quad then \quad  $w_a \in D_{n,n-\ell}$}. 
$$
For example, if $n = 6$, $r=10$, and $a = (2,1,3,1,6,2, 6,1, 3, 1)$, then
\begin{align*}
a & = (\underbrace{2,1,3,1,6,{\bf 2}, {\bf 6},1, {\bf 3}, {\bf 1}}_{\text{$\ell = 4$ distinct entries}})
\qquad \Rightarrow \qquad w_a   = (
\underbrace{4,5}_{n-\ell=2}, \underbrace{2,6,3,1}_{\ell=4}) \in D_{6,2},  \\
\shape(a) & = \shape(2,1,3,1,6,{ 2}, { 6},1, { 3}, { 1})  = 
\underbrace{ \{1,6\} \cup \{ 2,4,8,10\} \cup \{3,9\} \cup \{5,7\} }_{\ell=4 \text{ parts}}. 
\end{align*}
Note that the number of possible parts $\ell$ in $\shape(a)$ is bounded both by the number $r$ of subscripts and the number $n$ of possible choices of $a_i$. 

For a fixed set partition $K$ with $\ell$ parts and a fixed permutation $w \in D_{n,n-\ell}$ we can easily reconstruct the unique sequence $a \in \{1,\ldots, n\}^r$ with $\ell$ distinct entries such that $\shape(a) = K$ and $w_a = w$.  Thus, if we let ${\cal P}_r^\ell$ be the set partitions of $\{1, \ldots, r \}$ with $\ell$ parts, then
\begin{align*}
\sum_{a\in \{1, \ldots, n\}^r} q^{{\imaj}(w_a)} 
& = \sum_{\ell = 1}^{\min(r,n)} \sum_{K \in {\cal P}_r^\ell}  \sum_{a \in \{1,\ldots, n\}^r \atop \shape(a) = K} q^{{\imaj}(w_a)} 
 = \sum_{\ell = 1}^{\min(n,r)} \sum_{K \in {\cal P}_r^\ell}  \sum_{w \in D_{n,n-\ell}} q^{{\imaj}(w)} \\
& = \sum_{\ell = 1}^{\min(n,k)} S(r,\ell)  \sum_{w\in D_{n,n-\ell}} q^{{\imaj}(w)},
\end{align*}
where the last equality comes from the fact that the Stirling number $S(r,\ell)$ is the number of partitions of $\{1, \ldots, r\}$ into $\ell$ parts. 

To finish the proof of the proposition, we will show that 
\begin{equation}\label{FirstReduction}
\sum_{w \in D_{n,t}} q^{{\imaj}(w)}  = [n][n-1] \cdots [t + 1], \qquad 0 \le t < n.
\end{equation}
The shape of a permutation  $w$ is the composition  $\mu = (\mu_1, \ldots, \mu_\ell)$ of $n$ where $\mu_1$ is the first position $i$ where $w_i > w_{i+1}$, $\mu_1 + \mu_2$ is the next position $i$ where $w_i > w_{i+1}$ and so on.  
The sum in \eqref{FirstReduction} is over all partitions whose shape $\mu$ satisfies $\mu_1 \ge t$.  An inversion in a permutation $w$ is a pair $(i,j)$ such that $i < j$ and $w_i > w_j$ and $\inv(w)$ is the number of inversions in $w$.   Foata and Sch\"utzenberger \cite{FS} (see also \cite[Theorem 11.4.4]{Lo}) prove that the number of permutations of shape $\mu$ having $m$ inversions equals the number of permutations of shape $\mu$ having $m$ backsteps.  Thus, 
\begin{equation}
\sum_{w \in D_{n,t}} q^{{\imaj}(w)}  = \sum_{w \in D_{n,t}} q^{\inv(w)}. 
\end{equation}
Now, our coset representatives $D_{n,t}$ for $S_n/S_t$ are chosen with minimal length, so if $u \in D_{n,t}$ and $v \in S_t$, then $\inv(uv) = \inv(u) + \inv(v)$.  Thus, 
$$
[n]! = \sum_{s \in S_n} q^{\inv(w)} = \sum_{u \in D_{n,t}} \sum_{v \in S_t} q^{\inv (uv)} = \sum_{u \in D_{n,t}} q^{\inv(u)} \sum_{v \in S_t} q^{\inv (v)} =   \sum_{u \in D_{n,t}} q^{\inv(u)} [t]!,
$$
where the first and last equalities come from the  well-known result of MacMahon (see \cite[Cor 1.3.10]{Sta1}) that $\sum_{w \in S_n} q^{\inv(w)} = [n]!$.  Equation \eqref{FirstReduction} follows by dividing by $[t]!$ and replacing $\inv$ with $\imaj$. 
\end{proof}

For $n,r \in \ZZ_{>0}$, define,
\begin{equation}\label{PolynomialDefinition}
d_{n,r}(q) = \sum_{\ell = 1}^n  S(r,\ell)  [n] [n-1] \cdots [n-\ell+1].
\end{equation}
The first few values of $d_{n,r}(q)$, for increasing $r$,  are given by
\begin{align*}
d_{n,0}(q)  & = 1, \\
d_{n,1}(q)  & = [n],  \\
d_{n,2}(q) & =  [n]([n-1]+1), \\
d_{n,3}(q) & =  [n] (1 + 3 [n-1] + [n-1] [n-2]), \\
d_{n,4}(q) &=  [n](1 + 7 [n-1] + 6 [n-1][n-2] + [n-1][n-2][n-3]). 
\end{align*}
When $q=0$, we have $[j]_{q=0} = 1$, so $d_{n,r}(0) = \sum_{\ell = 0}^n S(r,\ell)$, which equals the $r$th Bell number $B(r)$ if $n\ge r$ and which is the number of set partitions of $\{1, \ldots, r\}$ into at most $n$ subsets if $n < r$.   When $q=1$ the sum in Proposition \ref{propsum} shows that $d_{n,r}(1)$ equals the cardinality of $\{1, \ldots, n\}^r$, so $d_{n,r}(1) = n^r$. 
These polynomials are tantalizingly close to those in the following identity of Garsia and Remmel \cite[I.17]{GR}
$$
\sum_{\ell = 1}^r S(r,\ell;q) [n][n-1] \cdots [n-r+1] = [n]^r,
$$
where $S(r,\ell;q)$ is a $q$-analog of the Stirling number $S(r,\ell)$. Like $d_{n,r}(q)$, these Garsia--Remmel polynomials specialize at $q=1$ to $n^r$, but they are different at $q=0$, since $[n]^r\vert_{q=0} =1$.

The next Corollary follows immediately from Corollary \ref{corsum} and Proposition \ref{propsum}.

\begin{cor}  For $n,r \in \ZZ_{>0}$, we have
$$
d_{n,r}(q) = \sum_{\ell = 1}^n  S(r,\ell)  [n] [n-1] \cdots [n-\ell+1] = \sum_{\lambda\in \Lambda_r^n} f^\lambda_n(q) m^\lambda_r.
$$
\end{cor}

\end{subsection}

\end{section}

\begin{section}{A Basis for the $\IR$ Module for $\GL_n(\FF_q)$}

In this section we construct a module $\IR_q^r$ for the finite general linear group $\GL_n(\FF_q)$ using $r$ iterations of Harish-Chandra restriction and induction.  We find a basis for $\IR_q^r$ that is indexed by  $q$-set partitions of $\{1, \ldots, r\}$.  It is easy to see that the number of these is  the polynomial $d_{n,r}(q)$ which appeared in Section 2, and so
$\dim(\IR_q^r) = d_{n,r}(q)$. The module $\IR_q^r$ is the defining space for the $q$-partition algebra, which will be analyzed in a subsequent paper by T.\ Halverson, A.\ Ram, and N.\ Thiem.

\subsection{A family of $q$-analogues to set partitions}\label{SectionqSetPartitions}

Let
$$\ZZ_{n}^r=\{(k_1,k_2,\ldots, k_r)\in \ZZ^r\mid 0\leq k_1,k_2,\ldots,k_r\leq n-1\},$$
which we can think of as a configuration of boxes stacked into an $(n-1)\times r$ rectangle.  That is, $(k_1,k_2,\ldots, k_r)$ denotes the collection of boxes with $k_j$ boxes stacked in the $j$th column.  For example,
$$
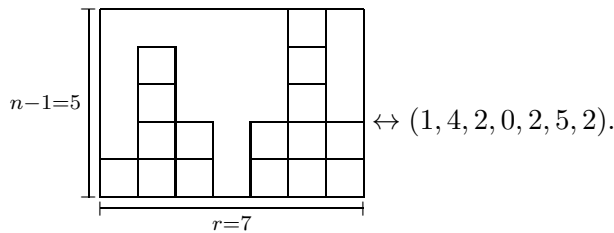
\xy<0cm,1.5cm>\xymatrix@R=.5cm@C=.5cm{
*={} \ar @{|-|} @<-.15cm> [ddddd]_{n-1=5} & *={} & *={} & *={} & *={} & *={} & *={} & *={} \ar @{-}[lllllll]\\
*={} & *={} & *={} \ar @{-} [l] & *={} & *={} & *={} & *={} \ar @{-} [l] & *={}\\
*={} & *={} & *={} \ar @{-} [l]  & *={} & *={}  & *={} & *={} \ar @{-} [l] & *={} \\
*={} & *={} & *={} \ar @{-} [l]  & *={} \ar @{-} [l]   & *={} & *={} \ar @{-} [l] & *={} \ar @{-} [l] & *={} \ar @{-} [l] \\
*={} & *={} \ar @{-} [l] & *={} \ar @{-} [l] & *={}  \ar @{-} [l]  & *={}  & *={} \ar @{-} [l] & *={} \ar @{-} [l] & *={}\ar @{-} [l] \\
*={}  \ar @{-} [uuuuu] \ar @{|-|} @<-.15cm> [rrrrrrr]_{r=7} & *={} \ar @{-} [uuuu]  & *={} \ar @{-} [uuuu]  & *={}  \ar @{-} [uu] & *={} \ar@{-} [uu]  & *={} \ar @{-} [uuuuu]  & *={} \ar @{-} [uuuuu]  & *={}\ar @{-}[lllllll] \ar @{-} [uuuuu] }\endxy\leftrightarrow (1,4,2,0,2,5,2).$$
Let
$$\mathcal{P}_{n\times r}=\{(k_1,k_2,\ldots, k_r)\in\ZZ_n^r \mid\ k_1=0, k_j=h\text{ implies $k_i=h-1$ for some $i<j$}\}.$$
We have a surjection
$$\begin{array}{ccc} \ZZ_n^r & \longrightarrow & \mathcal{P}_{n\times r}\\ k  & \mapsto & k^*,\end{array}\qquad\text{where $k_1^*=0$ and $k_j^*=\min\big\{k_j,\max\{k_1^*+1,k_2^*+1,\ldots,k_{j-1}^*+1\}\big\}$,}$$
which sends
$$k=
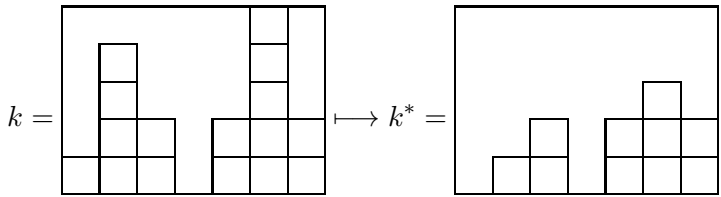
\xy<0cm,1.5cm>\xymatrix@R=.5cm@C=.5cm{
*={} & *={} & *={} & *={} & *={} & *={} & *={} & *={} \ar @{-}[lllllll]\\
*={} & *={} & *={} \ar @{-} [l] & *={} & *={} & *={} & *={} \ar @{-} [l] & *={}\\
*={} & *={} & *={} \ar @{-} [l]  & *={} & *={}  & *={} & *={} \ar @{-} [l] & *={} \\
*={} & *={} & *={} \ar @{-} [l]  & *={} \ar @{-} [l]   & *={} & *={} \ar @{-} [l] & *={} \ar @{-} [l] & *={} \ar @{-} [l] \\
*={} & *={} \ar @{-} [l] & *={} \ar @{-} [l] & *={}  \ar @{-} [l]  & *={}  & *={} \ar @{-} [l] & *={} \ar @{-} [l] & *={}\ar @{-} [l] \\
*={}  \ar @{-} [uuuuu]  & *={} \ar @{-} [uuuu]  & *={} \ar @{-} [uuuu]  & *={}  \ar @{-} [uu] & *={} \ar@{-} [uu]  & *={} \ar @{-} [uuuuu]  & *={} \ar @{-} [uuuuu]  & *={}\ar @{-}[lllllll] \ar @{-} [uuuuu] }\endxy \longmapsto
k^*=\xy<0cm,1.5cm>\xymatrix@R=.5cm@C=.5cm{
*={} & *={} & *={} & *={} & *={} & *={} & *={} & *={} \ar @{-}[lllllll]\\
*={} & *={} & *={}  & *={} & *={} & *={} & *={}& *={}\\
*={} & *={} & *={} & *={} & *={}  & *={} & *={}  \ar @{-} [l]  & *={}\\
*={} & *={} & *={}  & *={} \ar @{-} [l]   & *={} & *={} \ar @{-} [l] & *={} \ar @{-} [l] & *={} \ar @{-} [l] \\
*={} & *={} & *={} \ar @{-} [l] & *={}  \ar @{-} [l]  & *={}  & *={} \ar @{-} [l] & *={} \ar @{-} [l] & *={}\ar @{-} [l] \\
*={}  \ar @{-} [uuuuu]  & *={} \ar @{-} [u]  & *={} \ar @{-} [uu]  & *={}  \ar @{-} [uu] & *={} \ar@{-} [uu]  & *={} \ar @{-} [uuu]  & *={} \ar @{-} [uuu]  & *={}\ar @{-}[lllllll] \ar @{-} [uuuuu] }\endxy
$$
We will refer to $k^*$ as the \emph{$\ast$-height} of $k$.

There is a bijection,
$$\begin{array}{ccc} \mathcal{P}_{n\times r}&\longleftrightarrow&\left\{\begin{array}{c}\text{Set partitions of $\{1,2,\ldots,r\}$}\\ \text{with at most $n$ parts}\end{array}\right\} \\ k=(k_1,k_2,\ldots, k_r) & \mapsto & K_k,\end{array} $$
where $i$ and $j$ are in the same part of $K_k$ if and only if $k_i=k_j$.  That is,
$$
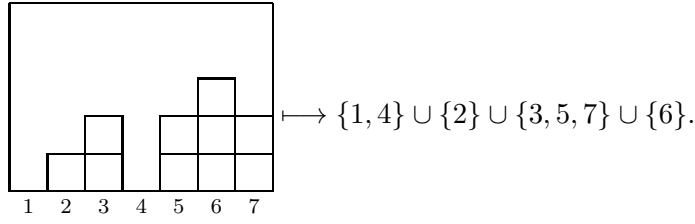
\xy<0cm,1.5cm>\xymatrix@R=.5cm@C=.5cm{
*={} & *={} & *={} & *={} & *={} & *={} & *={} & *={} \ar @{-}[lllllll]\\
*={} & *={} & *={}  & *={} & *={} & *={} & *={}& *={}\\
*={} & *={} & *={} & *={} & *={}  & *={} & *={}  \ar @{-} [l]  & *={}\\
*={} & *={} & *={}  & *={} \ar @{-} [l]   & *={} & *={} \ar @{-} [l] & *={} \ar @{-} [l] & *={} \ar @{-} [l] \\
*={} & *={} & *={} \ar @{-} [l] & *={}  \ar @{-} [l]  & *={}  & *={} \ar @{-} [l] & *={} \ar @{-} [l] & *={}\ar @{-} [l] \\
*={}  \ar @{-} [uuuuu]  & *={} \ar @{-} [u]  \ar @{} [l]^{1} & *={} \ar @{-} [uu]  \ar @{} [l]^{2} & *={}  \ar @{-} [uu] \ar @{} [l]^{3} & *={} \ar@{-} [uu]  \ar @{} [l]^{4}  & *={} \ar @{-} [uuu]  \ar @{} [l]^{5} & *={} \ar @{-} [uuu] \ar @{} [l]^{6}  & *={}\ar @{-}[lllllll] \ar @{-} [uuuuu] \ar @{} [l]^{7} }\endxy\longmapsto \{1,4\}\cup\{2\}\cup\{3,5,7\}\cup\{6\}.$$

To obtain $q$-analogues, fill the boxes with elements of $\FF_q$.  Let
$$\ZZ_n^r(q)=\{((k_1,a^{(1)}),(k_2,a^{(2)}),\ldots,(k_r, a^{(r)}))\mid (k_1,k_2,\ldots,k_r)\in \ZZ_n^r,a^{(j)}\in \FF_q^{k_j}\},$$
For example,
$$
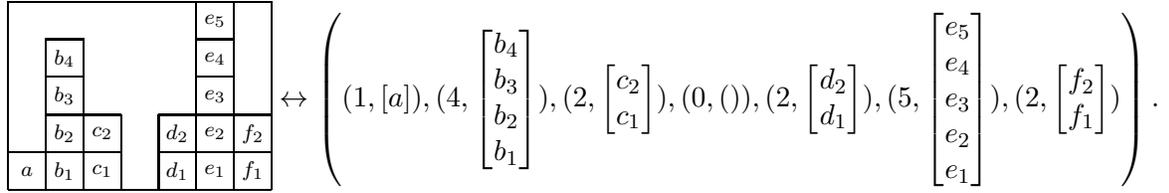
\xy<0cm,1.3cm>\xymatrix@R=.5cm@C=.5cm{
*={} & *={} & *={} & *={} & *={} & *={} & *={} & *={} \ar @{-}[lllllll]\\
*={} & *={} & *={} \ar @{-} [l] & *={} & *={} & *={} & *={} \ar @{-} [l] \ar @{} [ul]|{e_5}  & *={}\\
*={} & *={} & *={} \ar @{-} [l]  \ar @{} [ul]|{b_4} & *={} & *={}  & *={} & *={} \ar @{-} [l] \ar @{} [ul]|{e_4}  & *={}\\
*={} & *={} & *={} \ar @{-} [l]  \ar @{} [ul]|{b_3} & *={} \ar @{-} [l]   & *={} & *={} \ar @{-} [l] & *={} \ar @{-} [l] \ar @{} [ul]|{e_3} & *={} \ar @{-} [l] \\
*={} & *={} \ar @{-} [l] & *={} \ar @{-} [l] \ar @{} [ul]|{b_2} & *={}  \ar @{-} [l]  \ar @{} [ul]|{c_2} & *={}  & *={} \ar @{-} [l] \ar @{} [ul]|{d_2} & *={} \ar @{-} [l] \ar @{} [ul]|{e_2} & *={}\ar @{-} [l] \ar @{} [ul]|{f_2} \\
*={}  \ar @{-} [uuuuu]  & *={} \ar @{-} [uuuu] \ar @{} [ul]|{a}  & *={} \ar @{-} [uuuu]  \ar @{} [ul]|{b_1} & *={}  \ar @{-} [uu] \ar @{} [ul]|{c_1} & *={} \ar@{-} [uu]  & *={} \ar @{-} [uuuuu]  \ar @{} [ul]|{d_1} & *={} \ar @{-} [uuuuu]  \ar @{} [ul]|{e_1} & *={}\ar @{-}[lllllll] \ar @{-} [uuuuu]\ar @{} [ul]|{f_1}  }\endxy\leftrightarrow \left((1,[a]),(4,\left[\begin{array}{@{}c@{}} b_4\\b_3\\b_2\\b_1\end{array}\right]),(2,\left[\begin{array}{@{}c@{}}c_2\\ c_1\end{array}\right]),(0,()),(2,\left[\begin{array}{@{}c@{}}d_2\\ d_1\end{array}\right]),(5,\left[\begin{array}{@{}c@{}}e_5\\ e_4\\e_3\\e_2\\e_1\end{array}\right]),(2,\left[\begin{array}{@{}c@{}}f_2\\ f_1\end{array}\right])\right).$$
The $q$-analogue of set partitions of $\{1,2,\ldots, r\}$ with at most $n$ parts is the set
$$\mathcal{P}_{n\times r}(q)=\{((k_1,a^{(1)}),(k_2,a^{(2)}),\ldots,(k_r, a^{(r)}))\mid (k_1,k_2,\ldots,k_r)\in \ZZ_n^r, a^{(j)}\in \FF_q^{k_j-k_j^*}\}.$$
For example,
$$
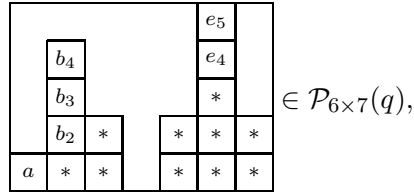
\xy<0cm,1.3cm>\xymatrix@R=.5cm@C=.5cm{
*={} & *={} & *={} & *={} & *={} & *={} & *={} & *={} \ar @{-}[lllllll]\\
*={} & *={} & *={} \ar @{-} [l] & *={} & *={} & *={} & *={} \ar @{-} [l] \ar @{} [ul]|{e_5}  & *={}\\
*={} & *={} & *={} \ar @{-} [l]  \ar @{} [ul]|{b_4} & *={} & *={}  & *={} & *={} \ar @{-} [l] \ar @{} [ul]|{e_4}  & *={}\\
*={} & *={} & *={} \ar @{-} [l]  \ar @{} [ul]|{b_3} & *={} \ar @{-} [l]   & *={} & *={} \ar @{-} [l] & *={} \ar @{-} [l] \ar @{} [ul]|{\ast} & *={} \ar @{-} [l] \\
*={} & *={} \ar @{-} [l] & *={} \ar @{-} [l] \ar @{} [ul]|{b_2} & *={}  \ar @{-} [l]  \ar @{} [ul]|{\ast} & *={}  & *={} \ar @{-} [l] \ar @{} [ul]|{\ast} & *={} \ar @{-} [l] \ar @{} [ul]|{\ast} & *={}\ar @{-} [l] \ar @{} [ul]|{\ast} \\
*={}  \ar @{-} [uuuuu]  & *={} \ar @{-} [uuuu] \ar @{} [ul]|{a}  & *={} \ar @{-} [uuuu]  \ar @{} [ul]|{\ast} & *={}  \ar @{-} [uu] \ar @{} [ul]|{\ast} & *={} \ar@{-} [uu]  & *={} \ar @{-} [uuuuu]  \ar @{} [ul]|{\ast} & *={} \ar @{-} [uuuuu]  \ar @{} [ul]|{\ast} & *={}\ar @{-}[lllllll] \ar @{-} [uuuuu]\ar @{} [ul]|{\ast}  }\endxy\in \mathcal{P}_{6\times 7}(q),$$
where the boxes labeled by $\ast$ give the $\ast$-height for the associated element in $\ZZ_n^r$.

An \emph{$n$-restricted $q$-set partition of $\{1,2,\ldots, r\}$} is an element of $\mathcal{P}_{n\times r}(q)$.  Given a set partition $K_h$ with $\ell$ parts, there are 
$$[n][n-1]\cdots[n-\ell+1]$$
different $n$-restricted $q$-set partitions of $\{1,2,\ldots, r\}$ with $\ast$-height $h$.  Thus,
$$|\mathcal{P}_{n\times r}(q)|=\sum_{\ell = 1}^n S(r,\ell) [n] [n-1] \cdots [n-\ell+1]=d_{n,r}(q),$$
where $d_{n,r}(q)$ is defined in (\ref{PolynomialDefinition}).
By the constructions of this section, we also easily obtain the specializations,
\begin{align*}
|\mathcal{P}_{n\times r}(0)|&=B(r), \qquad \text{for $n\geq r$,}\\
|\mathcal{P}_{n\times r}(1)|&=|\ZZ_n^r(1)|=n^r,
\end{align*}
where $B(r)$ is the $r$th Bell number.

\begin{subsection} {The Chevalley group $\GL_n(\FF_q)$}

The general linear group $G_n=\GL_n(\FF_q)$ has a double coset decomposition given by
\begin{equation}\label{UGBDecomposition}
G_n=\bigsqcup_{w\in S_n} U_B wB_n,
\end{equation}
where $S_n$ is the subgroup of permutation matrices, and 
$$B_n=\left\{\left(\begin{array}{ccc}  \ast & &\ast  \\ & \ddots  &  \\  0 &  & \ast \end{array}\right)\right\}\subseteq G_n\quad\text{and} \quad U_B=\left\{\left(\begin{array}{ccc}  1 & &\ast  \\ & \ddots  &  \\  0 &  & 1 \end{array}\right)\right\}\subseteq B_n$$
are the subgroups of upper-triangular matrices and unipotent upper-triangular matrices, respectively.  For $1\leq i<j\leq n$ and $a\in \FF_q$, let $x_{ij}(a)\in U_B$ be the matrix with $a$ in the $(i,j)$th position, ones on the diagonal, and zeroes everywhere else.   Note that for $i<j$, $k<l$, $a,b\in \FF_q$,
\begin{equation}\label{xxAction}
x_{ij}(a)x_{kl}(b)=\left\{\begin{array}{ll} x_{kl}(b)x_{il}(ab)x_{ij}(a), & \text{if $j=k$,}\\ x_{kl}(b)x_{kj}(-ab)x_{ij}(a), & \text{if $i=l$,}\\ x_{ij}(a+b), & \text{if $i=k$, $j=l$,}\\ x_{kl}(b)x_{ij}(a), & \text{otherwise.}\end{array}\right.
\end{equation}
For $w\in S_n$, we have
\begin{equation} \label{xwAction}
x_{ij}(a) w= w x_{w^{-1}(i)w^{-1}(j)}(a).
\end{equation}

Let 
$$L_n=\left\{\left(\begin{array}{c|ccc} G_{1}  & & 0 & \\  \hline & & &  \\  0 & & G_{n-1} & \\  &  & & \end{array}\right)\right\}\subseteq G_n,\quad U_n=\left\{\left(\begin{array}{c|ccc}  1 & \ast &\cdots & \ast\\ \hline &  & &  \\  0 & & \mathrm{Id}_{n-1} &   \\  &  & & \end{array}\right)\right\}\subseteq G_n,$$
and
$$P_n=L_nU_n=\left\{\left(\begin{array}{c|ccc}  G_1 &\ast & \cdots & \ast\\  \hline& &  & \\  0 & & G_{n-1} & \\ &  & & \end{array}\right)\right\}.$$

For $1\leq k\leq n-1$, let
$$w_k=s_ks_{k-1}\cdots s_1,$$
where $s_i$ is the simple reflection that switches $i$ and $i+1$.  By convention, $w_0=1$.  Note that the $w_k, 0 \le k \le n-1,$ give a set of minimal-length coset representatives for $S_n/(S_1\times S_{n-1})$.  For $1\leq k\leq n-1$ and $a=(a_1,a_2,\ldots,a_k)\in \FF_q^k$, let
$$w_k(a)=s_k(a_k)s_{k-1}(a_{k-1})\cdots s_1(a_1),\qquad\text{where} \qquad s_i(a_i)=x_{i,i+1}(a_i)s_i.$$
Then the decomposition
\begin{equation} \label{GPDecomposition}
G_n=\bigsqcup_{0\leq k\leq n-1\atop a\in \FF_q^k} w_k(a) P_n
\end{equation}
follows from (\ref{UGBDecomposition}) and (\ref{xwAction}).

\end{subsection}

\subsection{Harish-Chandra Restriction and Induction}

To make the notation more manageable, in this section we will assume that $n$ is fixed and drop the subscripts in $G_n$, $P_n$, $U_n$, $L_n$.  Let
$$e_U=\frac{1}{|U|}\sum_{u\in U} u,$$
so that $xe_U=e_U=e_Ux$ for all $x\in U$.
Since $U$ is a normal subgroup in $P$, there is a surjection $P\rightarrow P/U\cong L$, which gives rise to adjoint functors, called inflation and deflation, respectively,
\begin{align*}
&\begin{array}{rccc} \mathrm{Inf}_L^P:&\{\text{Left $L$-modules}\}&\longrightarrow &\{\text{Left $P$-modules}\},\\
 & V & \mapsto & e_UV\end{array}\\
 &\begin{array}{rccc} \mathrm{Def}_L^P:&\{\text{Left $P$-modules}\}&\longrightarrow &\{\text{Left $L$-modules}\}.\\
 & V & \mapsto & e_UV\end{array}
\end{align*}
By composing with induction and restriction, we obtain two functors
\begin{align*}
&\begin{array}{rccccc} \mathrm{Indf}_L^P:&\{\text{Left $L$-modules}\}&\longrightarrow &\{\text{Left $P$-modules}\} & \longrightarrow& \{\text{Left $G$-modules}\},\\
 & V & \mapsto & e_UV & \mapsto & \CC G\otimes_{\CC P} e_U V\end{array}\\
 &\begin{array}{rccccc} \mathrm{Resf}_L^P:&\{\text{Left $G$-modules}\}&\longrightarrow &\{\text{Left $P$-modules}\}& \longrightarrow &\{\text{Left $L$-modules}\}.\\
 & V & \mapsto &V & \mapsto & e_UV\end{array}
\end{align*}
Let $\One$ denote the trivial module of $G$.  Define the $G$-module
\begin{equation}
\IR_q^r=\big(\Indf_L^G\Resf_L^G\big)^r(\One), \qquad\text{for $r\geq 0$,}
\end{equation}
and the $L$-module
\begin{equation}
\IR^{r+\frac{1}{2}}_q=\Resf_L^G\big(\Indf_L^G\Resf_L^G\big)^r(\One), \qquad\text{for $r\geq 0$.}
\end{equation}

\begin{subsection} {A Basis for $\IR_q^r$}

Let 
$$\otimes_U= \otimes_{\CC P} e_U$$
denote tensoring over $\CC P$ and multiplying by $e_U$.
By construction it is clear that
\begin{align*}
\IR_q^r &= \CC\spanning\{g_1\otimes_U g_2\otimes_U \cdots \otimes_U g_r\otimes \One\mid g_1,g_2,\ldots, g_r\in G\}\\
&=\CC\spanning\{w_{k_1}(a^{(1)})\otimes_U\cdots\otimes_Uw_{k_r}(a^{(r)})\otimes \One\mid 0\leq k_1,\ldots, k_r\leq n-1, a^{(m)}\in \FF_q^{k_m}\},\\
\IR^{r+1/2}_q&=\CC\spanning\{e_Ug_1\otimes_U g_2\otimes_U \cdots \otimes_U g_r\otimes \One\mid g_1,g_2,\ldots, g_r\in G\}\\
&=\CC\spanning\{e_Uw_{k_1}(a^{(1)})\otimes_U\cdots\otimes_Uw_{k_r}(a^{(r)})\otimes \One\mid 0\leq k_1,\ldots, k_r\leq n-1, a^{(m)}\in \FF_q^{k_m}\}.
\end{align*}
However, these sets are generally not linearly independent.  The following lemma characterizes when two vectors are equal.

\begin{lemma}\label{VectorIdentification}
Fix $0\leq k_1,k_2,\ldots, k_l\leq n-1$ and $a^{(m)}\in \FF_q^{k_m}$.  Let $1\leq i<j<n$ and $t\in \FF_q^\times$.  Then
\begin{equation}\label{eUEffect}
w_{k_1}(a^{(1)})\otimes_U  \cdots \otimes_U w_{k_l}(a^{(l)}) \otimes_U x_{ij}(t) =w_{k_1}(a^{(1)})\otimes_U  \cdots \otimes_U w_{k_l}(a^{(l)}) \otimes_U 1
\end{equation}
if and only if $i=1$ or
there exists $1<m\leq l$ such that
$$
w_{k_m} w_{k_{m+1}}\cdots w_{k_{l}} \quad \text{ sends\ \  $i$ \ to \  1}.
$$
\end{lemma}
\begin{proof} Note that for $i=1$, $e_Ux_{1j}(t)=e_U$ and for $i>1$, $e_U x_{ij}(t)=x_{ij}(t) e_U$ (since $U$ is normal in $U_B$).  It therefore follows from (\ref{xxAction}) and (\ref{xwAction}) that for $0\leq k\leq n-1$ and $a\in \FF_q^k$,
\begin{equation}\label{waxRelation}
w_k(a)\otimes_U x_{ij}(t)=\left\{\begin{array}{ll} 
x_{w_k(i)j}(t)w_k(a)\otimes_U 1, & \text{if $k+1<i<j$,}\\
x_{w_k(i),j}(t)w_k(a)\otimes_U 1, & \text{if $1<i\leq k+1<j$,}\\
x_{w_k(i),j-1}(t)x_{w_k(i),k+1}(-a_{j-1}t)w_k(a)\otimes_U 1, & \text{if $1<i<j\leq k+1$,}\\
w_k(a)\otimes_U 1, & \text{if $i=1$.}\end{array}\right.
\end{equation}
If $i=1$, then (\ref{eUEffect}) follows.  If $i>1$ and there exists $1<m<l$ such that
$w_{k_m} w_{k_{m+1}}\cdots w_{k_{l}}$ sends $i$ to $1$, then using (\ref{waxRelation}) we can push $x_{ij}(t)$ left past the $w_{k}(a)$ so that
$$e_Uw_{k_m}(a^{(m)})\otimes_U\cdots \otimes_U w_{k_l}(a^{(l)})\otimes_U x_{ij}(t) =
e_Uuw_{k_m}(a^{(m)})\otimes_U\cdots   \otimes_U w_{k_l}(a^{(l)})\otimes_U 1,$$
where  $u\in U$ (since $w_{k_m}\cdots w_{k_l}(i)=1$).  But $e_Uu=e_U$, giving (\ref{eUEffect}).

Conversely, note that $w_k(a)\otimes_U x_{ij}(t)=w_k(a)\otimes_U 1$ if and only if $x_{ij}(t)\in U$.  Thus, if (\ref{eUEffect})
is true, then as we push $x_{ij}(t)$ to the left, at some point we must have 
\begin{align*}
w_{k_1}(a^{(1)})\otimes_U & \cdots \otimes_U w_{k_l}(a^{(l)}) \otimes_U x_{ij}(t)\\
 &= w_{k_1}(a^{(1)})\otimes_U  \cdots \otimes_Uw_{k_{m-1}}(a^{(m-1)})\otimes_Uu w_{k_m}(a^{(m)}\otimes_U\cdots \otimes_U w_{k_l}(a^{(l)}) \otimes_U 1.
 \end{align*}
for some $u\in U$.  By (\ref{waxRelation}) this can only happen if $w_{k_m} w_{k_{m+1}}\cdots w_{k_{l}}(i)=1$ for some $1<m\leq l$.
\end{proof}

Combinatorially, we associate a column of labeled boxes to $w_k(a)$, 
\begin{equation}\label{ColumnToVector}
\xy<0cm,2cm>\xymatrix@R=.5cm@C=.5cm{*={} \ar @{|-|} @<-.15cm> [dddddd]_{n-1}& *={} \ar @{-} [l]  \\
*={} &  *={} \\
*={}  & *={} \ar @{-} [l] \\
*={}  & *={}\ar @{-} [l] \ar @{} [ul]|{a_k} \\
*={}  & *={}  \ar @{-} [l] \save[]+<-.3cm,.35cm> *{\scriptscriptstyle\vdots}\restore\\
*={} & *={} \ar @{-} [l] \ar @{} [ul]|{a_2} \\
*={}  \ar @{-} [uuuuuu]   & *={}\ar @{-}[l] \ar @{-} [uuuuuu] \ar @{} [ul]|{a_1}}\endxy\quad \longleftrightarrow\quad w_k(a)=s_{k}(a_{k})\cdots s_2(a_2)s_1(a_1).
\end{equation}
We obtain vectors in $\IR_q^r$ by labeling $r$ stacks of boxes.  For example,
$$\xy<0cm,2cm>\xymatrix@R=.5cm@C=.5cm{
*={} & *={} & *={} & *={} & *={} & *={} & *={} & *={} \ar @{-}[lllllll]\\
*={} & *={} & *={} & *={} & *={} & *={} & *={} \ar @{-} [l] & *={}\\
*={} & *={} & *={} & *={} & *={} & *={} & *={} \ar @{-} [l] \ar @{} [ul]|{e_6} & *={}\\
*={} & *={} & *={} \ar @{-} [l] & *={} & *={} & *={} & *={} \ar @{-} [l] \ar @{} [ul]|{e_5} & *={}\\
*={} & *={} & *={} \ar @{-} [l] \ar @{} [ul]|{b_4} & *={} & *={}  & *={} & *={} \ar @{-} [l] \ar @{} [ul]|{e_4} & *={}\ar @{-} [l] \\
*={} & *={} & *={} \ar @{-} [l] \ar @{} [ul]|{b_3} & *={} & *={} \ar @{-} [l]  & *={} \ar @{-} [l] & *={} \ar @{-} [l] \ar @{} [ul]|{e_3} & *={} \ar @{-} [l]\ar @{} [ul]|{f_3} \\
*={} & *={} \ar @{-} [l] & *={} \ar @{-} [l] \ar @{} [ul]|{b_2} & *={}  & *={} \ar @{-} [l] \ar @{} [ul]|{c_2} & *={} \ar @{-} [l] \ar @{} [ul]|{d_2} & *={} \ar @{-} [l] \ar @{} [ul]|{e_2} & *={}\ar @{-} [l] \ar @{} [ul]|{f_2}\\
*={}  \ar @{-} [uuuuuuu]  & *={} \ar @{-} [uuuu]  \ar @{} [ul]|{a_1}& *={} \ar @{-} [uuuu]  \ar @{} [ul]|{b_1} & *={}  \ar @{-} [uu] & *={} \ar@{-} [uu] \ar @{} [ul]|{c_1} & *={} \ar @{-} [uuuuuu] \ar @{} [ul]|{d_1} & *={} \ar @{-} [uuuuuu] \ar @{} [ul]|{e_1}  & *={}\ar @{-}[lllllll] \ar @{-} [uuuuuuu] \ar @{} [ul]|{f_1}}\endxy
\ \longleftrightarrow\ w_1(a)\otimes_U w_4(b)\otimes_U w_0\otimes_U w_2(c)\otimes_U w_2(d) \otimes_U w_6(e)\otimes_U w_3(f)\otimes\One.$$
Lemma \ref{VectorIdentification} implies that not all choices of the vectors $a,b,c,d,e,f$ will give different basis vectors of $W_q^7$.  In our example, any change to the $\ast$-ed values in
$$
\xy<0cm,2cm>\xymatrix@R=.5cm@C=.5cm{
*={} & *={} & *={} & *={} & *={} & *={} & *={} & *={} \ar @{-}[lllllll]\\
*={} & *={} & *={} & *={} & *={} & *={} & *={} \ar @{-} [l] & *={}\\
*={} & *={} & *={} & *={} & *={} & *={} & *={} \ar @{-} [l] \ar @{} [ul]|{e_6} & *={}\\
*={} & *={} & *={} \ar @{-} [l] & *={} & *={} & *={} & *={} \ar @{-} [l] \ar @{} [ul]|{e_5} & *={}\\
*={} & *={} & *={} \ar @{-} [l] \ar @{} [ul]|{b_4} & *={} & *={}  & *={} & *={} \ar @{-} [l] \ar @{} [ul]|{e_4} & *={}\ar @{-} [l] \\
*={} & *={} & *={} \ar @{-} [l] \ar @{} [ul]|{b_3} & *={} & *={} \ar @{-} [l]  & *={} \ar @{-} [l] & *={} \ar @{-} [l] \ar @{} [ul]|{\scriptstyle e_3^*} & *={} \ar @{-} [l]\ar @{} [ul]|{\scriptstyle f_3^*} \\
*={} & *={} \ar @{-} [l] & *={} \ar @{-} [l] \ar @{} [ul]|{b_2} & *={}  & *={} \ar @{-} [l] \ar @{} [ul]|{\scriptstyle c_2^*} & *={} \ar @{-} [l] \ar @{} [ul]|{\scriptstyle d_2^*} & *={} \ar @{-} [l] \ar @{} [ul]|{\scriptstyle e_2^*} & *={}\ar @{-} [l] \ar @{} [ul]|{\scriptstyle f_2^*}\\
*={}  \ar @{-} [uuuuuuu]  & *={} \ar @{-} [uuuu]  \ar @{} [ul]|{a_1}& *={} \ar @{-} [uuuu]  \ar @{} [ul]|{\scriptstyle b_1^*} & *={}  \ar @{-} [uu] & *={} \ar@{-} [uu] \ar @{} [ul]|{\scriptstyle c_1^*} & *={} \ar @{-} [uuuuuu] \ar @{} [ul]|{\scriptstyle d_1^*} & *={} \ar @{-} [uuuuuu] \ar @{} [ul]|{\scriptstyle e_1^*}  & *={}\ar @{-}[lllllll] \ar @{-} [uuuuuuu] \ar @{} [ul]|{\scriptstyle f_1^*} }\endxy
$$
will not change the vector in $\IR^7$.   That is, given an element of $\ZZ_{n}^r(q)$, Lemma \ref{VectorIdentification} implies that the $\ast$-height determines which entries can have arbitrary values (see Section \ref{SectionqSetPartitions}).  In particular, if the $k_l$th entry
\begin{align*}
w_{k_1}(a^{(1)})\otimes_U &\cdots \otimes_U w_{k_{l-1}}(a^{(l-1)}) \otimes_U w_{k_l}(b) \otimes_U\cdots\\ &=  w_{k_1}(a^{(1)})\otimes_U \cdots \otimes_U w_{k_{l-1}}(a^{(l-1)})\otimes_U  x_{1,k_l+1}(b_1) \cdots x_{k_l,k_l+1}(b_{k_l})w_{k_l} \otimes_U \cdots
\end{align*}
has $\ast$-height $h$ then for each $1\leq i\leq h$, there exists $1<m_i\leq l-1$ such that 
$$w_{k_{m_i}}\cdots w_{k_{l-1}}(i)=1.$$
By Lemma \ref{VectorIdentification} these entries can be replaced by arbitrary entries, and for these entries we average over all possible choices.  Thus, for each element in $\mathcal{P}_{n\times r}(q)$, we obtain a basis vector.  Specifically, for $0\leq k^*\leq k\leq n-1$ and $a=(a_1,\cdots, a_{k-k^*})\in \FF_q^{k-k^*}$, associate
\begin{equation}\label{StarColumnToVector}
\xy<0cm,2cm>\xymatrix@R=.5cm@C=.5cm{*={} \ar @{|-|} @<-.5cm> [dddddddd]_{n-1}& *={} \ar @{-} [l]  \\
*={} & *={} \\
*={}  \ar @{|-|} @<-.15cm> [dddddd]_{k} &  *={}\ar @{-} [l] \\
*={} &  *={} \ar @{-} [l] \save[]+<.7cm,.3cm> *{\scriptstyle a_{k-k^*}} \ar @/^{.2cm}/ []-<.25cm,-.25cm>\restore\\
*={} & *={} \ar @{-} [l] \save[]+<-.3cm,.35cm> *{\scriptscriptstyle\vdots}\restore\\
*={}  & *={}\ar @{-} [l] \ar @{} [ul]|{a_1}  \ar @{|-|} @<.15cm> [ddd]^{k^*} \\
*={}  & *={} \ar @{-} [l]\ar @{} [ul]|{\ast} \\
*={} & *={} \ar @{-} [l]  \save[]+<-.3cm,.35cm> *{\scriptscriptstyle\vdots}\restore\\
*={}  \ar @{-} [uuuuuuuu]   & *={}\ar @{-}[l] \ar @{-} [uuuuuuuu] \ar @{} [ul]|{\ast}}\endxy \longleftrightarrow\quad \overset{\ast}{w}_k(a)=s_{k}(a_{k-k*})\cdots s_{k^*+1}(a_1)\frac{1}{q^{k^*}}\sum_{b\in \FF_q^{k^*}} s_{k^*}(b_{k^*})\cdots s_1(b_1).
\end{equation}
For $K=((k_1,a^{(1)}),\ldots,(k_r,a^{(r)}))\in \mathcal{P}_{n\times r}(q)$, let
$$v_K=\overset{\ast}{w}_{k_1}(a^{(1)})\otimes_U \overset{\ast}{w}_{k_2}(a^{(2)})\otimes_U\cdots\otimes_U \overset{\ast}{w}_{k_r}(a^{(r)})\in \CC G\otimes_U \CC G\otimes_U\cdots \otimes_U \CC G\otimes_U\One_n.$$
For example,
\begin{align*} v_{\xymatrix@R=.25cm@C=.25cm{
*={} & *={} & *={} & *={} & *={} & *={} & *={} & *={} \ar @{-}[lllllll]\\
*={} & *={} & *={} & *={} & *={} & *={} & *={} \ar @{-} [l] & *={}\\
*={} & *={} & *={} & *={} & *={} & *={} & *={} \ar @{-} [l] \ar @{} [ul]|{\scriptscriptstyle g} & *={}\\
*={} & *={} & *={} \ar @{-} [l] & *={} & *={} & *={} & *={} \ar @{-} [l] \ar @{} [ul]|{\scriptscriptstyle f} & *={}\\
*={} & *={} & *={} \ar @{-} [l] \ar @{} [ul]|{\scriptscriptstyle d} & *={} & *={}  & *={} & *={} \ar @{-} [l] \ar @{} [ul]|{\scriptscriptstyle e} & *={}\ar @{-} [l] \\
*={} & *={} & *={} \ar @{-} [l] \ar @{} [ul]|{\scriptscriptstyle c} & *={} & *={} \ar @{-} [l]  & *={} \ar @{-} [l] & *={} \ar @{-} [l] \ar @{} [ul]|{\ast} & *={} \ar @{-} [l]\ar @{} [ul]|{\ast} \\
*={} & *={} \ar @{-} [l] & *={} \ar @{-} [l] \ar @{} [ul]|{\scriptscriptstyle b} & *={}  & *={} \ar @{-} [l] \ar @{} [ul]|{\ast} & *={} \ar @{-} [l] \ar @{} [ul]|{\ast} & *={} \ar @{-} [l] \ar @{} [ul]|{\ast} & *={}\ar @{-} [l] \ar @{} [ul]|{\ast}\\
*={}  \ar @{-} [uuuuuuu]  & *={} \ar @{-} [uuuu]  \ar @{-} [ul]|{\scriptscriptstyle{a}}& *={} \ar @{-} [uuuu]  \ar @{} [ul]|{\ast} & *={}  \ar @{-} [uu] & *={} \ar@{-} [uu] \ar @{} [ul]|{\ast} & *={} \ar @{-} [uuuuuu] \ar @{} [ul]|{\ast} & *={} \ar @{-} [uuuuuu] \ar @{} [ul]|{\ast}  & *={}\ar @{-}[lllllll] \ar @{-} [uuuuuuu] \ar @{} [ul]|{\ast}}}
=&x_{12}(a)w_1\otimes x_{45}(d)x_{35}(c)x_{25}(b)\frac{1}{q}\sum_{t\in \FF_q}x_{15}(t)w_4\otimes 1
 \otimes\frac{1}{q^2} \sum_{s,t\in \FF_q} x_{23}(s)x_{13}(t)w_2\\
&\otimes \frac{1}{q^2}\sum_{s,t\in\FF_q^\times} x_{23}(s)x_{13}(t)w_2
\otimes x_{67}(g)x_{57}(f)x_{47}(e)\frac{1}{q^3}\hspace{-.15cm} \sum_{r,s,t\in \FF_q} \hspace{-.15cm} x_{37}(r)x_{27}(s)x_{17}(t) w_6\\
&\otimes\frac{1}{q^3} \sum_{r,s,t\in \FF_q} x_{34}(r)x_{24}(s)x_{14}(t)w_3\otimes \One_n.\end{align*}

Lemma \ref{VectorIdentification} and the following discussion imply that the $v_K$ are linearly independent, so we have proved the first part of the following theorem.  

\begin{thm} \label{BasisTheorem} Let $r\in \ZZ_{\geq 0}$.  Then
\begin{enumerate}
\item[(a)]  The $G$-module $\IR_q^r$ has a basis given by
$$\{ v_K\mid K\in\mathcal{P}_{n\times r}(q)\},$$
and thus $\dim(\IR_q^r)=d_{n,r}(q)$.
\item[(b)]  The $L$-module $W_q^{r+1/2}$ has a basis given by
$$\{ v_K\mid K\in \mathcal{P}_{n\times r+1}(q)\text{ with $k_1=0$}\},$$
and thus $\dim(\IR_q^r)=d_{n,r+1}(q)/[n]$.
\end{enumerate}
\end{thm}

To prove Theorem \ref{BasisTheorem} (b), it suffices to characterize what happens in $e_U\IR_q^r$.  Let
$$\begin{array}{ccc}\mathcal{P}_{n\times r}(q) & \longrightarrow & \{K\in \mathcal{P}_{n\times r+1}(q)\mid k_1=0\}\\
K & \mapsto & \tilde{K}\end{array}$$
be the surjective function given by the following algorithm.
\begin{enumerate}
\item[(1)] Add an empty column to the left side of $K$ and set $m=1$,
\item[(2)] If the resulting diagram is in $\mathcal{P}_{n\times r+1}(q)$, stop.  Else set $m:=m+1$.
\item[(3)] If column $m$ has an unstarred box, then replace the bottom unstarred entry by $\ast$.  Go to step (2).
\end{enumerate}
For example, we get
$$\xy<0cm,2cm>\xymatrix@R=.5cm@C=.5cm{
*={} & *={} & *={} & *={} & *={} & *={} & *={} & *={} \ar @{-}[lllllll]\\
*={} &*={} & *={} & *={} & *={} & *={} & *={} \ar @{-} [l] & *={}\\
*={} & *={} & *={} & *={} & *={} & *={} & *={} \ar @{-} [l] \ar @{} [ul]|{e_3} & *={}\\
*={} & *={} & *={} \ar @{-} [l] & *={} & *={} & *={} & *={} \ar @{-} [l] \ar @{} [ul]|{e_2} & *={}\\
*={} & *={}  & *={} \ar @{-} [l] \ar @{} [ul]|{b_3} & *={} & *={} \ar @{-} [l] & *={} & *={} \ar @{-} [l] \ar @{} [ul]|{e_1} & *={}\ar @{-} [l] \\
*={} & *={} \ar @{-} [l]  & *={} \ar @{-} [l] \ar @{} [ul]|{b_2} & *={} & *={} \ar @{-} [l] \ar @{} [ul]|{c_1} & *={} \ar @{-} [l] & *={} \ar @{-} [l] \ar @{} [ul]|{\ast} & *={} \ar @{-} [l]\ar @{} [ul]|{\ast} \\
*={} & *={} \ar @{-} [l]  \ar @{} [ul]|{a_2} & *={} \ar @{-} [l] \ar @{} [ul]|{b_1} & *={}  & *={} \ar @{-} [l] \ar @{} [ul]|{\ast} & *={} \ar @{-} [l] \ar @{} [ul]|{\ast} & *={} \ar @{-} [l] \ar @{} [ul]|{\ast} & *={}\ar @{-} [l] \ar @{} [ul]|{\ast}\\
*={} \ar @{-} [uuuuuuu]  & *={} \ar @{-} [uuuu]  \ar @{} [ul]|{a_1} & *={} \ar @{-} [uuuu]  \ar @{} [ul]|{\ast} & *={}  \ar @{-} [uuu] & *={} \ar@{-} [uuu] \ar @{} [ul]|{\ast} & *={} \ar @{-} [uuuuuu] \ar @{} [ul]|{\ast} & *={} \ar @{-} [uuuuuu] \ar @{} [ul]|{\ast}  & *={}\ar @{-}[lllllll] \ar @{-} [uuuuuuu] \ar @{} [ul]|{\ast}}\endxy
\overset{m=1}{\rightarrow}
\xy<0cm,2cm>\xymatrix@R=.5cm@C=.5cm{
*={} & *={} & *={} & *={} & *={} & *={} & *={} & *={} & *={} \ar @{-}[llllllll]\\
*={} & *={} &*={} & *={} & *={} & *={} & *={} & *={} \ar @{-} [l] & *={}\\
*={} & *={} & *={} & *={} & *={} & *={} & *={} & *={} \ar @{-} [l] \ar @{} [ul]|{e_3} & *={}\\
*={} & *={} & *={} & *={} \ar @{-} [l] & *={} & *={} & *={} & *={} \ar @{-} [l] \ar @{} [ul]|{e_2} & *={}\\
*={} & *={} & *={}  & *={} \ar @{-} [l] \ar @{} [ul]|{b_3} & *={} & *={} \ar @{-} [l] & *={} & *={} \ar @{-} [l] \ar @{} [ul]|{e_1} & *={}\ar @{-} [l] \\
*={} & *={} & *={} \ar @{-} [l]  & *={} \ar @{-} [l] \ar @{} [ul]|{b_2} & *={} & *={} \ar @{-} [l] \ar @{} [ul]|{c_1} & *={} \ar @{-} [l] & *={} \ar @{-} [l] \ar @{} [ul]|{\ast} & *={} \ar @{-} [l]\ar @{} [ul]|{\ast} \\
*={} & *={} & *={} \ar @{-} [l]  \ar @{} [ul]|{a_2} & *={} \ar @{-} [l] \ar @{} [ul]|{b_1} & *={}  & *={} \ar @{-} [l] \ar @{} [ul]|{\ast} & *={} \ar @{-} [l] \ar @{} [ul]|{\ast} & *={} \ar @{-} [l] \ar @{} [ul]|{\ast} & *={}\ar @{-} [l] \ar @{} [ul]|{\ast}\\
*={} \ar @{-} [uuuuuuu] & *={} \ar @{-} [uu]   \ar @{} [l]^{\uparrow} & *={} \ar @{-} [uuuu]  \ar @{} [ul]|{a_1} & *={} \ar @{-} [uuuu]  \ar @{} [ul]|{\ast} & *={}  \ar @{-} [uuu] & *={} \ar@{-} [uuu] \ar @{} [ul]|{\ast} & *={} \ar @{-} [uuuuuu] \ar @{} [ul]|{\ast} & *={} \ar @{-} [uuuuuu] \ar @{} [ul]|{\ast}  & *={}\ar @{-}[llllllll] \ar @{-} [uuuuuuu] \ar @{} [ul]|{\ast}}\endxy
\overset{m=2}{\rightarrow}
\xy<0cm,2cm>\xymatrix@R=.5cm@C=.5cm{
*={} & *={} & *={} & *={} & *={} & *={} & *={} & *={} & *={} \ar @{-}[llllllll]\\
*={} & *={} &*={} & *={} & *={} & *={} & *={} & *={} \ar @{-} [l] & *={}\\
*={} & *={} & *={} & *={} & *={} & *={} & *={} & *={} \ar @{-} [l] \ar @{} [ul]|{e_3} & *={}\\
*={} & *={} & *={} & *={} \ar @{-} [l] & *={} & *={} & *={} & *={} \ar @{-} [l] \ar @{} [ul]|{e_2} & *={}\\
*={} & *={} & *={}  & *={} \ar @{-} [l] \ar @{} [ul]|{b_3} & *={} & *={} \ar @{-} [l] & *={} & *={} \ar @{-} [l] \ar @{} [ul]|{e_1} & *={}\ar @{-} [l] \\
*={} & *={} & *={} \ar @{-} [l]  & *={} \ar @{-} [l] \ar @{} [ul]|{b_2} & *={} & *={} \ar @{-} [l] \ar @{} [ul]|{c_1} & *={} \ar @{-} [l] & *={} \ar @{-} [l] \ar @{} [ul]|{\ast} & *={} \ar @{-} [l]\ar @{} [ul]|{\ast} \\
*={} & *={} & *={} \ar @{-} [l]  \ar @{} [ul]|{a_2} & *={} \ar @{-} [l] \ar @{} [ul]|{b_1} & *={}  & *={} \ar @{-} [l] \ar @{} [ul]|{\ast} & *={} \ar @{-} [l] \ar @{} [ul]|{\ast} & *={} \ar @{-} [l] \ar @{} [ul]|{\ast} & *={}\ar @{-} [l] \ar @{} [ul]|{\ast}\\
*={} \ar @{-} [uuuuuuu] & *={} \ar @{-} [uu]  & *={} \ar @{-} [uuuu]  \ar @{} [ul]|{\ast} \ar @{} [l]^{\uparrow} & *={} \ar @{-} [uuuu]  \ar @{} [ul]|{\ast}  & *={}  \ar @{-} [uuu] & *={} \ar@{-} [uuu] \ar @{} [ul]|{\ast} & *={} \ar @{-} [uuuuuu] \ar @{} [ul]|{\ast} & *={} \ar @{-} [uuuuuu] \ar @{} [ul]|{\ast}  & *={}\ar @{-}[llllllll] \ar @{-} [uuuuuuu] \ar @{} [ul]|{\ast}}\endxy
$$
$$\overset{m=3}{\rightarrow}
\xy<0cm,2cm>\xymatrix@R=.5cm@C=.5cm{
*={} & *={} & *={} & *={} & *={} & *={} & *={} & *={} & *={} \ar @{-}[llllllll]\\
*={} & *={} &*={} & *={} & *={} & *={} & *={} & *={} \ar @{-} [l] & *={}\\
*={} & *={} & *={} & *={} & *={} & *={} & *={} & *={} \ar @{-} [l] \ar @{} [ul]|{e_3} & *={}\\
*={} & *={} & *={} & *={} \ar @{-} [l] & *={} & *={} & *={} & *={} \ar @{-} [l] \ar @{} [ul]|{e_2} & *={}\\
*={} & *={} & *={}  & *={} \ar @{-} [l] \ar @{} [ul]|{b_3} & *={} & *={} \ar @{-} [l] & *={} & *={} \ar @{-} [l] \ar @{} [ul]|{e_1} & *={}\ar @{-} [l] \\
*={} & *={} & *={} \ar @{-} [l]  & *={} \ar @{-} [l] \ar @{} [ul]|{b_2} & *={} & *={} \ar @{-} [l] \ar @{} [ul]|{c_1} & *={} \ar @{-} [l] & *={} \ar @{-} [l] \ar @{} [ul]|{\ast} & *={} \ar @{-} [l]\ar @{} [ul]|{\ast} \\
*={} & *={} & *={} \ar @{-} [l]  \ar @{} [ul]|{a_2} & *={} \ar @{-} [l] \ar @{} [ul]|{\ast} & *={}  & *={} \ar @{-} [l] \ar @{} [ul]|{\ast} & *={} \ar @{-} [l] \ar @{} [ul]|{\ast} & *={} \ar @{-} [l] \ar @{} [ul]|{\ast} & *={}\ar @{-} [l] \ar @{} [ul]|{\ast}\\
*={} \ar @{-} [uuuuuuu] & *={} \ar @{-} [uu]  & *={} \ar @{-} [uuuu]  \ar @{} [ul]|{\ast} & *={} \ar @{-} [uuuu]  \ar @{} [ul]|{\ast} \ar @{} [l]^{\uparrow} & *={}  \ar @{-} [uuu]   & *={} \ar@{-} [uuu] \ar @{} [ul]|{\ast} & *={} \ar @{-} [uuuuuu] \ar @{} [ul]|{\ast} & *={} \ar @{-} [uuuuuu] \ar @{} [ul]|{\ast}  & *={}\ar @{-}[llllllll] \ar @{-} [uuuuuuu] \ar @{} [ul]|{\ast}}\endxy
\overset{m=4}{\rightarrow}
\xy<0cm,2cm>\xymatrix@R=.5cm@C=.5cm{
*={} & *={} & *={} & *={} & *={} & *={} & *={} & *={} & *={} \ar @{-}[llllllll]\\
*={} & *={} &*={} & *={} & *={} & *={} & *={} & *={} \ar @{-} [l] & *={}\\
*={} & *={} & *={} & *={} & *={} & *={} & *={} & *={} \ar @{-} [l] \ar @{} [ul]|{e_3} & *={}\\
*={} & *={} & *={} & *={} \ar @{-} [l] & *={} & *={} & *={} & *={} \ar @{-} [l] \ar @{} [ul]|{e_2} & *={}\\
*={} & *={} & *={}  & *={} \ar @{-} [l] \ar @{} [ul]|{b_3} & *={} & *={} \ar @{-} [l] & *={} & *={} \ar @{-} [l] \ar @{} [ul]|{e_1} & *={}\ar @{-} [l] \\
*={} & *={} & *={} \ar @{-} [l]  & *={} \ar @{-} [l] \ar @{} [ul]|{b_2} & *={} & *={} \ar @{-} [l] \ar @{} [ul]|{c_1} & *={} \ar @{-} [l] & *={} \ar @{-} [l] \ar @{} [ul]|{\ast} & *={} \ar @{-} [l]\ar @{} [ul]|{\ast} \\
*={} & *={} & *={} \ar @{-} [l]  \ar @{} [ul]|{a_2} & *={} \ar @{-} [l] \ar @{} [ul]|{\ast} & *={}  & *={} \ar @{-} [l] \ar @{} [ul]|{\ast} & *={} \ar @{-} [l] \ar @{} [ul]|{\ast} & *={} \ar @{-} [l] \ar @{} [ul]|{\ast} & *={}\ar @{-} [l] \ar @{} [ul]|{\ast}\\
*={} \ar @{-} [uuuuuuu] & *={} \ar @{-} [uu]  & *={} \ar @{-} [uuuu]  \ar @{} [ul]|{\ast} & *={} \ar @{-} [uuuu]  \ar @{} [ul]|{\ast} & *={}  \ar @{-} [uuu] \ar @{} [l]^{\uparrow}  & *={} \ar@{-} [uuu] \ar @{} [ul]|{\ast} & *={} \ar @{-} [uuuuuu] \ar @{} [ul]|{\ast} & *={} \ar @{-} [uuuuuu] \ar @{} [ul]|{\ast}  & *={}\ar @{-}[llllllll] \ar @{-} [uuuuuuu] \ar @{} [ul]|{\ast}}\endxy
\overset{m=5}{\rightarrow}
\xy<0cm,2cm>\xymatrix@R=.5cm@C=.5cm{
*={} & *={} & *={} & *={} & *={} & *={} & *={} & *={} & *={} \ar @{-}[llllllll]\\
*={} & *={} &*={} & *={} & *={} & *={} & *={} & *={} \ar @{-} [l] & *={}\\
*={} & *={} & *={} & *={} & *={} & *={} & *={} & *={} \ar @{-} [l] \ar @{} [ul]|{e_3} & *={}\\
*={} & *={} & *={} & *={} \ar @{-} [l] & *={} & *={} & *={} & *={} \ar @{-} [l] \ar @{} [ul]|{e_2} & *={}\\
*={} & *={} & *={}  & *={} \ar @{-} [l] \ar @{} [ul]|{b_3} & *={} & *={} \ar @{-} [l] & *={} & *={} \ar @{-} [l] \ar @{} [ul]|{e_1} & *={}\ar @{-} [l] \\
*={} & *={} & *={} \ar @{-} [l]  & *={} \ar @{-} [l] \ar @{} [ul]|{b_2} & *={} & *={} \ar @{-} [l] \ar @{} [ul]|{\ast} & *={} \ar @{-} [l] & *={} \ar @{-} [l] \ar @{} [ul]|{\ast} & *={} \ar @{-} [l]\ar @{} [ul]|{\ast} \\
*={} & *={} & *={} \ar @{-} [l]  \ar @{} [ul]|{a_2} & *={} \ar @{-} [l] \ar @{} [ul]|{\ast} & *={}  & *={} \ar @{-} [l] \ar @{} [ul]|{\ast} & *={} \ar @{-} [l] \ar @{} [ul]|{\ast} & *={} \ar @{-} [l] \ar @{} [ul]|{\ast} & *={}\ar @{-} [l] \ar @{} [ul]|{\ast}\\
*={} \ar @{-} [uuuuuuu] & *={} \ar @{-} [uu]  & *={} \ar @{-} [uuuu]  \ar @{} [ul]|{\ast} & *={} \ar @{-} [uuuu]  \ar @{} [ul]|{\ast} & *={}  \ar @{-} [uuu] & *={} \ar@{-} [uuu] \ar @{} [ul]|{\ast} \ar @{} [l]^{\uparrow} & *={} \ar @{-} [uuuuuu] \ar @{} [ul]|{\ast}  & *={} \ar @{-} [uuuuuu] \ar @{} [ul]|{\ast}  & *={}\ar @{-}[llllllll] \ar @{-} [uuuuuuu] \ar @{} [ul]|{\ast}}\endxy$$
$$\overset{m=6}{\rightarrow}
\xy<0cm,2cm>\xymatrix@R=.5cm@C=.5cm{
*={} & *={} & *={} & *={} & *={} & *={} & *={} & *={} & *={} \ar @{-}[llllllll]\\
*={} & *={} &*={} & *={} & *={} & *={} & *={} & *={} \ar @{-} [l] & *={}\\
*={} & *={} & *={} & *={} & *={} & *={} & *={} & *={} \ar @{-} [l] \ar @{} [ul]|{e_3} & *={}\\
*={} & *={} & *={} & *={} \ar @{-} [l] & *={} & *={} & *={} & *={} \ar @{-} [l] \ar @{} [ul]|{e_2} & *={}\\
*={} & *={} & *={}  & *={} \ar @{-} [l] \ar @{} [ul]|{b_3} & *={} & *={} \ar @{-} [l] & *={} & *={} \ar @{-} [l] \ar @{} [ul]|{e_1} & *={}\ar @{-} [l] \\
*={} & *={} & *={} \ar @{-} [l]  & *={} \ar @{-} [l] \ar @{} [ul]|{b_2} & *={} & *={} \ar @{-} [l] \ar @{} [ul]|{\ast} & *={} \ar @{-} [l] & *={} \ar @{-} [l] \ar @{} [ul]|{\ast} & *={} \ar @{-} [l]\ar @{} [ul]|{\ast} \\
*={} & *={} & *={} \ar @{-} [l]  \ar @{} [ul]|{a_2} & *={} \ar @{-} [l] \ar @{} [ul]|{\ast} & *={}  & *={} \ar @{-} [l] \ar @{} [ul]|{\ast} & *={} \ar @{-} [l] \ar @{} [ul]|{\ast} & *={} \ar @{-} [l] \ar @{} [ul]|{\ast} & *={}\ar @{-} [l] \ar @{} [ul]|{\ast}\\
*={} \ar @{-} [uuuuuuu] & *={} \ar @{-} [uu]  & *={} \ar @{-} [uuuu]  \ar @{} [ul]|{\ast} & *={} \ar @{-} [uuuu]  \ar @{} [ul]|{\ast} & *={}  \ar @{-} [uuu] & *={} \ar@{-} [uuu] \ar @{} [ul]|{\ast} & *={} \ar @{-} [uuuuuu] \ar @{} [ul]|{\ast}  \ar @{} [l]^{\uparrow}  & *={} \ar @{-} [uuuuuu] \ar @{} [ul]|{\ast} & *={}\ar @{-}[llllllll] \ar @{-} [uuuuuuu] \ar @{} [ul]|{\ast}}\endxy
\overset{m=7}{\rightarrow}
\xy<0cm,2cm>\xymatrix@R=.5cm@C=.5cm{
*={} & *={} & *={} & *={} & *={} & *={} & *={} & *={} & *={} \ar @{-}[llllllll]\\
*={} & *={} &*={} & *={} & *={} & *={} & *={} & *={} \ar @{-} [l] & *={}\\
*={} & *={} & *={} & *={} & *={} & *={} & *={} & *={} \ar @{-} [l] \ar @{} [ul]|{e_3} & *={}\\
*={} & *={} & *={} & *={} \ar @{-} [l] & *={} & *={} & *={} & *={} \ar @{-} [l] \ar @{} [ul]|{e_2} & *={}\\
*={} & *={} & *={}  & *={} \ar @{-} [l] \ar @{} [ul]|{b_3} & *={} & *={} \ar @{-} [l] & *={} & *={} \ar @{-} [l] \ar @{} [ul]|{\ast} & *={}\ar @{-} [l] \\
*={} & *={} & *={} \ar @{-} [l]  & *={} \ar @{-} [l] \ar @{} [ul]|{b_2} & *={} & *={} \ar @{-} [l] \ar @{} [ul]|{\ast} & *={} \ar @{-} [l] & *={} \ar @{-} [l] \ar @{} [ul]|{\ast} & *={} \ar @{-} [l]\ar @{} [ul]|{\ast} \\
*={} & *={} & *={} \ar @{-} [l]  \ar @{} [ul]|{a_2} & *={} \ar @{-} [l] \ar @{} [ul]|{\ast} & *={}  & *={} \ar @{-} [l] \ar @{} [ul]|{\ast} & *={} \ar @{-} [l] \ar @{} [ul]|{\ast} & *={} \ar @{-} [l] \ar @{} [ul]|{\ast} & *={}\ar @{-} [l] \ar @{} [ul]|{\ast}\\
*={} \ar @{-} [uuuuuuu] & *={} \ar @{-} [uu]  & *={} \ar @{-} [uuuu]  \ar @{} [ul]|{\ast} & *={} \ar @{-} [uuuu]  \ar @{} [ul]|{\ast} & *={}  \ar @{-} [uuu] & *={} \ar@{-} [uuu] \ar @{} [ul]|{\ast} & *={} \ar @{-} [uuuuuu] \ar @{} [ul]|{\ast} & *={} \ar @{-} [uuuuuu] \ar @{} [ul]|{\ast}  \ar @{} [l]^{\uparrow}  & *={}\ar @{-}[llllllll] \ar @{-} [uuuuuuu] \ar @{} [ul]|{\ast}}\endxy\overset{\text{Stop}}{\rightarrow}$$
\begin{lemma}
Let $K\in \mathcal{P}_{n\times k}(q)$.  Then $e_Uv_K=v_{\tilde{K}}$.
\end{lemma}
\begin{proof} Apply Lemma \ref{VectorIdentification} to the vector
$$w_0\otimes_U v_K$$
to obtain the statement of this lemma.
\end{proof}

\begin{proof}[Proof of Theorem \ref{BasisTheorem} (b)] We have that
\begin{align*}
\IR^{r+1/2}_q&=e_U\IR_q^r\\
&=\CC\spanning\{e_Uv_K\mid K\in \mathcal{P}_{n\times r}(q)\}\\
&=\CC\spanning\{v_{\tilde K}\mid \mathcal{P}_{n\times r}(q)\},
\end{align*}
and by Lemma \ref{VectorIdentification}, the vectors in the last set are linearly independent.
\end{proof}
\end{subsection}

\end{section}

\section{Group action on $\IR^r$}

In general,
$$gw_k(a)\otimes_U v=w_l(b)\otimes_U p v, \qquad \text{where $gw_k(a)=w_l(b)p$.}$$
Thus, globally the matrix of $g$ is the matrix of $g$ acting by left multiplication on $G/P$.   The group $G$ has generators given by 
$$\{x_{ij}(t)\mid 1\leq i<j\leq n,t\in \FF_q\}\cup\{s_1,s_2,\ldots, s_{n-1}\}\cup \{h_k(t)\mid 1\leq k\leq n, t\in \FF_q^\times\},$$
where $h_k(t)$ is the identity matrix with the $k$th diagonal $1$ replaced by $t$.  The generators of $G$ act on $\IR^r_q$ in the following way:
\begin{align*}
s_i w_k(a)\otimes_U v &=\left\{\begin{array}{ll} w_{k}(a)\otimes_U s_i v, & \text{if $i> k+1$,}\\
w_{k+1}({\scriptstyle a_1,a_2,\ldots,a_k,0})\otimes_U v, & \text{if $i=k+1$,}\\
w_{k-1}({\scriptstyle a_1,\ldots, a_{k-1}})\otimes_U v, & \text{if $i=k$, $a_k=0$,}\\
w_{k}({\scriptstyle a_1,\ldots,a_{k-1},a_k^{-1}})\otimes_U h_{k+1}(-a_k^{-1})v, & \text{if $i=k$, $a_k\neq 0$,}\\
w_{k}({\scriptstyle a_1,\ldots, a_{i-1}, a_{i+1},a_i,a_{i+2},\ldots,a_k})\otimes_U s_{i+1} v, & \text{if $i<k$.}\end{array}\right.\\
h_j(b)w_k(a)\otimes_U v&=
\left\{\begin{array}{ll} w_{k}(a)\otimes_U h_j(b) v, & \text{if $j> k+1$,}\\
w_{k}({\scriptstyle  a_1b^{-1},\ldots,a_kb^{-1}})\otimes_U  v, & \text{if $j=k+1$,}\\
w_{k}({\scriptstyle  a_1,\ldots, a_{j-1}, ba_{j},a_{j+1},\ldots,a_k})\otimes_U h_{j+1}(b) v, & \text{if $i<k+1$.}\end{array}\right.\\
x_{ij}(b)w_k(a)\otimes_U v &=\left\{
\begin{array}{ll} w_{k}(a)\otimes_U x_{ij}(b) v, & \text{if $i> k+1$,}\\
w_{k}(a)\otimes_U x_{k+1,j}(-a_kb)\cdots x_{2,j}(-a_1b) v, & \text{if $i=k+1\neq 1$,}\\
w_{k}(a)\otimes_U  v, & \text{if $i=k+1=1$,}\\
w_{k}(a)\otimes_U x_{i+1,j}(b)v, & \text{if $i<k+1<j$,}\\
w_{k}({\scriptstyle a_1,\ldots,a_{i-1},a_i+b,a_{i+1},\ldots, a_k})\otimes_U v, & \text{if $i<k+1=j$,}\\
w_{k}({\scriptstyle  a_1,\ldots, a_{i-1},a_i+ba_j, a_{i+1},\ldots,a_k})\otimes_U x_{i+1,j+1}(b) v, & \text{if $j<k+1$.}\end{array}\right.
\end{align*}

\end{document}